\documentclass[preprint,12pt]{elsarticle}
\usepackage{graphicx} 
\usepackage{latexsym} 
\usepackage{float}
\usepackage{geometry}
\usepackage{amsthm}
\usepackage{tikz}
\usepackage{bm}   
\usepackage{diagbox}
\usepackage[backref]{hyperref} 
\hypersetup{
hidelinks
}
\usepackage{CJK}
\usepackage{amsmath} 
\usepackage{subcaption}
\usepackage{appendix}  
\usepackage{indentfirst} 
\usepackage{xcolor}
\allowdisplaybreaks
\usepackage{amssymb} 
\newtheorem{theorem}{Theorem}
\newtheorem{lemma}{Lemma}
\newtheorem{remark}{Remark}

\begin{document}

\begin{frontmatter}
\title{Multiscale Partially Explicit Splitting with Mass Lumping for High-Contrast Wave Equations}
\author[WL]{Shu Fan Li}
\author[WL]{Wing Tat Leung \corref{cor1}}
\ead{wtleung27@cityu.edu.hk}

\cortext[cor1]{Corresponding author}
\address[WL]{Department of Mathematics, City University of Hong Kong.}

\begin{abstract}
    In this paper, contrast-independent partially explicit time discretization for wave equations in heterogeneous high-contrast media via mass lumping is concerned. 
    By employing a mass lumping scheme to diagonalize the mass matrix, the matrix inversion procedures can be avoided, thereby significantly enhancing computational efficiency especially in the explicit part. In addition, after decoupling the resulting system, higher order time discretization techniques can be applied to get better accuracy within the same time step size.
    Furthermore, the spatial space is divided into two components: contrast-dependent (“fast”) and contrast-independent (“slow”) subspaces. Using this decomposition, our objective is to introduce an appropriate time splitting method that ensures stability and guarantees contrast-independent discretization under suitable conditions. 
    We analyze the stability and convergence of the proposed algorithm. In particular, we discuss the second order central difference and higher order Runge–Kutta method for a wave equation. Several numerical examples are presented to confirm our theoretical results and to demonstrate that our proposed algorithm achieves high accuracy while reducing computational costs for high-contrast problems.
\end{abstract}
\end{frontmatter}
\section{Introduction}
Wave phenomena can be observed across diverse scientific and engineering disciplines, such as medical imaging, seismic inversion, and multiphase composite materials. Many of them are modeled using wave equations with highly oscillatory coefficients 
 \cite{masson2010finite,saenger2007finite,virieux1986p,cheung2021explicit}. Therein, material properties often exhibit multiscale characteristics and high-contrast, which typically include some features with distinct properties in narrow strips or channels. These heterogeneities bring challenges to numerical simulations. To address these issues, various types of coarse grid models have been developed. 

Multiscale methods are frequently utilized in various applications to reduce computational complexity and enable problem-solving on coarser grids. Numerous multiscale techniques have been developed, with some approaches formulating coarse grid problems using effective media properties derived from homogenization theory \cite{wu2002analysis,chung2018constraint,bourgeat1984homogenized,efendiev2003numerical}. For example, composition rules for constructing the finite element basis are presented in \cite{allaire2005multiscale}. However, these approaches often rely on pre-formulated assumptions \cite{durlofsky1991numerical,bakhvalov2012homogenisation}.
Alternatively, some approaches focus on constructing multiscale basis functions and formulating coarse grid equations. 
One prominent example is the multiscale finite element methods (MsFEM) \cite{jenny2003multi,chen2002multiscale,hou1999convergence}, with several variations of this method described in detail in \cite{efendiev2009multiscale}.
Other widely used multiscale approaches include generalized multiscale finite element methods (GMsFEM) \cite{efendiev2013generalized,chung2014adaptive,chung2016goal},
constraint energy minimizing GMsFEM (CEM-GMsFEM) \cite{chung2018constraint,chung2022contrast,cheung2018constraint},
nonlocal multi-continua approaches (NLMC) \cite{chung2018non,leung2022multirate,wang2023adaptive}, and localized orthogonal decomposition (LOD) \cite{henning2014localized,zhang2022combined}. 

Therein, in the presence of high-contrast, multiple basis functions are required within coarse elements. For instance, GMsFEM was proposed to extract local dominant modes through local spectral problems on coarse grids and represent the multiscale features by multiscale basis functions or continua \cite{efendiev2013generalized}. However, it is difficult to achieve mesh-dependent convergence without oversampling and a large number of basis functions.
CEM-GMsFEM was proposed to resolve this difficulty, which consists two main steps. First, it solves local spectral problems to construct auxiliary basis functions. Subsequently, constraint energy minimization problems are solved to derive localized basis functions. By combining eigenfunctions with energy-minimizing properties, the resulting basis functions exhibit exponential decay away from the target coarse element, allowing for computing basis functions locally. Additionally, the convergence depends on the coarse grid size when the oversampling domain is carefully chosen.

Our approach is based on some fundamental splitting algorithms \cite{chung2022contrast,chung2021contrast,hundsdorfer2013numerical,marchuk1990splitting}. These algorithms are initially constructed to split various processes in some coupled systems. Recently, numerous methods have been introduced to address spatial heterogeneities. Therein, explicit schemes have the advantage of rapid time marching \cite{virieux1984sh,gyrya2019explicit}, whereas implicit schemes typically provide unconditionally stability but at a higher computational cost \cite{alexander1977diagonally,burrage1979stability}.
To balance stability and efficiency, it is essential to develop a method that mitigates the contrast-dependent time step constraints while maintaining computational efficiency.
We combine the temporal splitting algorithms and spatial multiscale methods. The spatial space is divided into two components and subspaces are used in the temporal splitting. As a result, smaller systems are inverted at each time step, reducing computational cost. More precisely, some subspaces are treated explicitly in time, while others are treated implicitly. However, the proposed approach is still coupled through the mass matrix.

To address this issue, we further introduce a mass lumping scheme \cite{boscarino2013implicit,hu2025uniform,hinton1976note}. 
Given a mass lumping scheme via the diagonalization technique of the mass matrix, matrix inversion procedures can be avoided, significantly enhancing computational efficiency, especially in the explicit part. Specifically, once the mass matrix is lumped and diagonalized, the system becomes easier to solve without solving complex linear systems at each time step. In addition, after decoupling the resulting system, higher order time discretization techniques can be applied and get better accuracy under the same time steps.  


In the paper, we present numerical results for wave simulations in various heterogeneous media with differing source terms via mass lumping. We investigate and compare several results, showing that the stability condition is independent of contrast. Further, we demonstrate that after applying mass lumping, the error remains comparable to methods without mass lumping, but with significantly enhanced computational efficiency. Our numerical results show that the error decays as the number of oversampling layers grows, the number of eigenfunctions in each coarse element grows, the coarse grid size decreases, the time step size drops, and is also influenced by the choice of time discretization scheme.

The main findings in our paper are as follows
\begin{itemize}
  \item Partially explicit temporal splitting algorithms with spatial multiscale methods are designed for high-contrast wave equations. The stability and convergence conditions for the proposed method are presented, which is proven to be contrast-independent.
  \item Given a mass lumping scheme to decouple the system, the matrix inversion procedures are avoided, which can significantly reduce computational costs and use the higher order time discretization scheme separately.
  \item Piecewise constants in each coarse grid and matrix of $V_{\text{aux,1}}$ are employed for an upscaling system, identifying appropriate local problems to represent high-contrast areas.
  \item The numerical results can achieve high accuracy and the expected convergence rate, strongly confirming our theoretical findings.
\end{itemize}
The paper is organized as follows. Some notions are provided in Section \ref{002}. The details of the proposed method, including the construction of the global spaces, local spaces, the mass lumping scheme, and the corresponding systems are derived in Section \ref{003}. The stability and the convergence of the method will be analyzed in Section \ref{004}. In Section \ref{005}, we briefly introduce a partially explicit Runge-Kutta scheme for our system, and in Section \ref{006}, numerical simulations are given for certification. The conclusions are presented in Section \ref{007}.

\section{Preliminaries}\label{002}
We consider the second order wave equation in a heterogeneous domain
\begin{equation}\label{0.1}
    \frac{\partial^2 u}{\partial t^2}=\text{div}(\kappa\nabla u)+f~\mathrm{in}~[0,T]\times\Omega,
\end{equation}
where $f(x,t)$ is a given source term, $\Omega\subset\mathbb{R}^{d}$. $\kappa$ is a high-contrast with $\kappa_0 \leq \kappa(x)\leq\kappa_1$ where $\cfrac{\kappa_1}{\kappa_0}\gg 1$, and is a highly heterogeneous multiscale field. The above equation is subjected to the homogeneous Dirichlet boundary condition $u=0$ on $[0,T]\times \partial\Omega $ and the initial conditions $u(x,0)=u_0(x)$ and $u_t(x,0)=v_0(x)$ in $\Omega$.
The proposed method can be easily extended to other types of boundary conditions, for example, absorbing boundary conditions (ABC). 
Here, some notions are introduced. Denote $\mathcal{T}^H$ as a coarse-grid partition of $\Omega$ into finite elements, which does not necessarily resolve any multiscale features. The generic element $K$ in this partition $\mathcal{T}^H$ is referred to as a coarse element, and $H>0$ is the coarse grid size. Let the number of coarse grid nodes be $N_c$ and the number of coarse elements by $N$. Assume that each coarse element is partitioned into a connected union of fine grid blocks, forming the partition $\mathcal{T}^h$, where $h>0$ is called the fine grid size. It is assumed that the fine grid is sufficiently fine to resolve the solution. $V_h$ denotes the space spanned by the fine grid basis functions defined later. For the $i$-th coarse element $K_i$, let $V_h(K_i)$ be the conforming bilinear elements defined on the fine grid $\mathcal{T}^h$ within $K_i$. An illustration of the fine grid, coarse grid, and oversampling domain is shown in Figure \ref{fig:grid2}.

We write the problem (\ref{0.1}) in the fine grid weak formulation: find the solution $u_h(\cdot,t)\in V_h$ satisfying
  
\begin{equation}\label{0.2}
    \left(\frac{\partial^2 u_h}{\partial t^2},w\right)_{L^2(\Omega)}+a(u_h,w)=(f,w)_{L^2(\Omega)}~\mathrm{in}~[0,T]\times\Omega,
\end{equation}
where $u=0$ on $[0,T]\times \partial\Omega $ and the initial conditions $u(x,0)=u_h(0)$ and $u_t(x,0)=v_h(0)$ in $\Omega$.
Therein, the bilinear form $a(\cdot,\cdot)$ is given by 
$$a(u,v)=\int_{\Omega}\kappa\nabla u\cdot\nabla v,$$ and defines the energy norm $\|u\|_a=a(u,u)^{\frac{1}{2}}$.
The initial data is $L^2$ projected onto the finite element space $V_h$.
For the coarse grid space, we seek an approximation $u_H(t,\cdot)\in V_H$, where $V_H$ is a finite dimensional space ($H$ is a spatial mesh size), such that
\begin{equation}\label{0.21}
    \left(\frac{\partial^2 u_H}{\partial t^2},w\right)_{L^2(\Omega)}+a(u_H,w)=(f,w)_{L^2(\Omega)}~\mathrm{in}~[0,T]\times\Omega,
\end{equation}
where $(u_H(\cdot,0),v)=(u_0,v)$, $(u_{H,t}(\cdot,0),v)=(v_0,v)$ for all $v\in V_H$, .


  
\begin{figure}[htbp]
\centering
\begin{tikzpicture}[scale=0.5]
  \draw[step=0.5cm,gray,very thin] (-6,-6) grid (6,6);
  \draw[step=1cm,red,thin] (-6.5,-6.5) grid (6.5,6.5);
  \draw[step=2cm,blue,thick] (-6.5,-6.5) grid (6.5,6.5);
        \fill[olive!80!black, fill opacity=0.5] (-2,-2) rectangle (4,4); 
        \fill[yellow!90!orange, fill opacity=0.9] (0,0) rectangle (2,2); 
  \draw[ultra thick] (-2,-2) rectangle (4,4);

  \draw[fill=red] (3,3) rectangle (3.5,3.5);
  \node[font=\fontsize{8}{9.6}\selectfont\bfseries] at (3.98,4.4) {fine grid};
  \node[font=\fontsize{10}{12}\selectfont\bfseries] at (-3.15,-3.15) {$K_{i,1}$};
  \node[font=\fontsize{10}{12}\selectfont\bfseries] at (1,1) {$K_i$};

  \draw[->,thick] (3.9,4.2) -- (3.6,3.6);
  \draw[->,thick] (-3,-3) -- (-2.1,-2.1);
  
\end{tikzpicture}
\caption{Illustration of the fine grid, coarse grid $K_i$ and oversampling domain $K_{i,1}$.}
\label{fig:grid2}
\end{figure}
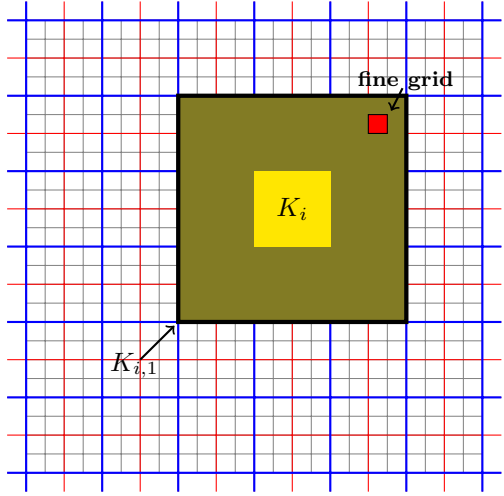

\section{The construction of the multiscale spaces}\label{003}
In this section, we discuss the construction of multiscale finite element spaces. The process is divided into two main parts. First, we construct the auxiliary spaces $V_{\text{aux,1}},V_{\text{aux,2}}$. Then, we utilize these auxiliary spaces to construct the multiscale spaces $V_{H,1},V_{H,2}$.  A detailed discussion on the construction of auxiliary multiscale basis functions, as well as global and local multiscale basis functions, is presented in the context of the wave equation.
\subsection{The construction of the auxiliary spaces}
To construct the auxiliary basis functions, we employ piecewise constant functions for $V_{\text{aux,1}}$ and adopt the GMsFEM framework for $V_{\text{aux,2}}$.
By utilizing spatial splitting based on a multiscale decomposition of the approximation space, the first spatial space is designed to account for spatial features. 
Thus, the auxiliary basis functions for $V_{\text{aux},1}$ are constructed as follows. For each coarse element $K_{i}\in \mathcal{T}^H$, we consider
$$\psi_j^{(i)}=I_{K_j^{(i)}}(x),$$
where $I_{K_j^{(i)}}$ is the characteristic function of $K_j^{(i)}\subset K_i$. Denote $K_1^{(i)}$ as the regions with low wave speed and $K_2^{(i)}$ as the regions with high wave speed. In this paper, we consider 
$$K_1^{(i)}=\{x\in K_i|\kappa(x)\leq \kappa_{\text{cutoff}}\},~K_2^{(i)}=\{x\in K_i|\kappa(x)> \kappa_{\text{cutoff}}\},$$ where $\kappa_{\text{cutoff}}$ is a constant defining the cutoff of the high wave speed and low wave speed regions. The number of basis functions in every coarse element is 1 or 2, denoted as $l_i$.

Next, we solve the following eigenvalue problem on every coarse element $K_i$ to obtain the auxiliary basis functions for $V_{\text{aux},2}$. Precisely, we find $(\xi_j^{(i)},\gamma_j^{(i)})\in V(K_i)\cap \tilde{V}\times \mathbb{R}$, such that
$$\int_{\omega_i}\kappa\nabla\xi_j^{(i)}\cdot\nabla v=\gamma_j^{(i)}\int_{\omega_i}\xi_j^{(i)}v~\forall v\in V(K_i)\cap\tilde{V},$$
where $\tilde{V}=\{v\in V|(v,\psi_j^{(i)})=0~\forall i,j\}$. The eigenvalues $\gamma_j^{(i)}$ are arranged in ascending order with respect to $j$, and we use the first $L_i$ eigenfunctions to construct $V_{\text{aux,2}}$.
Thus, the auxiliary spaces $V_{\text{aux},1}$ and $V_{\text{aux},2}$ are defined as 
$$V_{\text{aux},1}=\text{span}_{i,j}\Big\{\psi_j^{(i)}:1\leq j\leq l_i,1\leq i\leq N\Big\},~V_{\text{aux},2}=\text{span}_{i,j}\Big\{\xi_j^{(i)}:1\leq j\leq L_i,1\leq i\leq N\Big\}.$$
The local $L^2(K_i)$ projection operators $\pi_i:L^2(K_i)\rightarrow{V_{\text{aux},i}}~i=1,2$ can be defined as
$$\pi_1(v)=\sum_{j=1}^{l_i}\Big(v,\psi_j^{(i)}\Big)_{L^2(K_i)}\psi_j^{(i)},~\pi_2(v)=\sum_{j=1}^{L_i}\Big(v,\xi_j^{(i)}\Big)_{L^2(K_i)}\xi_j^{(i)}.$$

Since the coarse elements are disjoint, the auxiliary basis functions form an $L^2(\Omega)$ orthonormal basis functions for $V_{\text{aux,1}}$ and $V_{\text{aux,2}}$, i.e.,
\begin{equation}
    \begin{aligned}
        \left(\psi_j^{(i)},\psi_{j'}^{(i')}\right)_{L^2(\Omega)}=\delta_{i,i'}\delta_{j,j'}~\forall 1\leq j\leq l_i,1\leq j'\leq l_i' ~\text{and}~1\leq i,i'\leq N,
    \end{aligned}
\end{equation}
and
\begin{equation}
    \begin{aligned}
        \left(\xi_j^{(i)},\xi_{j'}^{(i')}\right)_{L^2(\Omega)}=\delta_{i,i'}\delta_{j,j'}~\forall 1\leq j\leq L_i,1\leq j'\leq L_i' ~\text{and}~1\leq i,i'\leq N.
    \end{aligned}
\end{equation}
The global $L^2$ projection operator $\pi:L^2(\Omega)\rightarrow{V_{\text{aux},1}\oplus V_{\text{aux},2}}$ is defined by $\pi=\sum_{i=1,2}\pi_{i}$.

\subsection{The construction of the global spaces}
Next, we construct our global basis functions and the global spaces by using the CEM-GMsFEM framework.

For each auxiliary basis function $\psi_{j}^{(i)}$, $\xi_j^{(i)}$, the global basis functions can be defined by 
finding $\phi_{glo,1,j}^i\in V_h$ and $\zeta_{j,1}^i\in V_{\text{aux,1}}\oplus V_{\text{aux,2}}$ such that
\begin{equation}\label{eq01}
\begin{aligned}
a\Big(\phi_{glo,1,j}^i,v\Big)+\Big(\zeta_{j,1}^i,v\Big)&=0~\forall v\in V_h,\\
\Big(\phi_{glo,1,j}^i-\psi_k^{(l)},\chi_1\Big)&=0~\forall \chi_1\in V_{\text{aux,1}},\\
\Big(\phi_{glo,1,j}^i,\xi_k^{(l)}\Big)&=0,
\end{aligned}
\end{equation}
and finding $\phi_{glo,2,j}^i\in V_h$ and $\zeta_{j,2}^i\in V_{\text{aux,1}}\oplus V_{\text{aux,2}}$ such that
\begin{equation}\label{eq02}
\begin{aligned}
    a\Big(\phi_{glo,2,j}^i,v\Big)+\Big(\zeta_{j,2}^i,v\Big)&=0~\forall v\in V_h,\\
    \Big(\phi_{glo,2,j}^i,\psi_k^{(l)}\Big)&=0,\\
    \Big(\phi_{glo,2,j}^i-\xi_k^{(l)},\chi_2\Big)&=0~\forall \chi_2\in V_{\text{aux,2}}.
\end{aligned}
\end{equation}
We use these global basis functions to construct our global spaces, which are defined as
$$V_{glo,1}=\text{span}\{\phi_{glo,1,j}^i:1\leq i\leq N,1\leq j\leq l_i\}, V_{glo,2}=\text{span}\{\phi_{glo,2,j}^i:1\leq i\leq N,1\leq j\leq L_i\}.$$
Then, the $L^2$ projections can be defined as $\pi|_{V_{glo,1}}:V_{glo,1}\rightarrow{V_{\text{aux,1}}\oplus V_{\text{aux,2}}},$ and $\pi|_{V_{glo,2}}:V_{glo,2}\rightarrow{V_{\text{aux,1}}\oplus V_{\text{aux,2}}}$.

Let $\tau$ be the time step size and $t_n=n\tau,~n=0,1,\dots,N,~T=N\tau.$ After constructing the global basis functions and global spaces, we first consider the second order central difference scheme. The fully discretized scheme is shown as follows
\begin{equation}\label{0.4}
\begin{aligned}
     &\left(\frac{u_{H}^{n+1}-2u_{H}^{n}+u_{H}^{n-1}}{\tau^2},w\right)+\frac{1}{2}a\left(u_{H,1}^{n+1}+u_{H,1}^{n-1}+2u_{H,2}^{n},w\right)=\left(f^n,w\right)~ \forall w \in V_{H,1}^{(m)},\\
     &\left(\frac{u_{H}^{n+1}-2u_{H}^{n}+u_{H}^{n-1}}{\tau^2},w\right)+a\left(\omega u_{H,1}^{n}+(1-\omega)\left(\frac{u_{H,1}^{n+1}+u_{H,1}^{n-1}}{2}\right)+u_{H,2}^{n},w\right)=\left(f^n,w\right)~\\&~~~~~~~\forall w \in V_{H,2}^{(m)}.
\end{aligned}
\end{equation}

However, it is noted that the system (\ref{0.4}) remains coupled via the first two terms. After using mass lumping, our system can be decoupled, and the two equations can be solved separately. Specifically speaking, the first equation solves for fast components implicitly and the second equation solves for slow components explicitly. This separation allows for efficient computation by handling the fast and slow dynamics with appropriate numerical methods tailored to their respective characteristics.

\subsection{The mass lumping scheme}
For $u_H=\sum_{i,j}u_{j,1}^{(i)}\phi_{glo,1,j}^{i}+u_{j,2}^{(i)}\phi_{glo,2,j}^{i},$ we have
$$a(u_H,v)=-\left(\sum_{i,j}u_{j,1}^{(i)}\zeta_{j,1}^i,v\right)-\left(\sum_{i,j}u_{j,2}^{(i)}\zeta_{j,2}^i,v\right)~\forall v\in V_{glo,1}+V_{glo,2}.$$
Then, the weak formulation of $u_{H,tt}=\nabla\cdot(\kappa\nabla u_H)+f$ is
\begin{equation}
    \left(\sum_{i,j}(u_{j,1}^{(i)})_{tt}\phi_{glo,1,j}^{i}+\sum_{i,j}(u_{j,2}^{(i)})_{tt}\phi_{glo,2,j}^{i},w\right)=\left(\sum_{i,j}u_{j,1}^{(i)}\zeta_{j,1}^i+\sum_{i,j}u_{j,2}^{(i)}\zeta_{j,2}^i+f,w\right),
\end{equation}
for all $w\in W_H$, and $W_H$ denotes a testing space. Consider $W_H$ to be $(V_{\text{aux},1}\oplus V_{\text{aux},2})$ and it derives
\begin{equation}\label{ml1}
\begin{aligned}
    \left(\sum_{i,j}(u_{j,1}^{(i)})_{tt}\phi_{glo,1,j}^{i}+\sum_{i,j}(u_{j,2}^{(i)})_{tt}\phi_{glo,2,j}^{i},\psi_k^{(l)}\right)=\left(\sum_{i,j}u_{j,1}^{(i)}\zeta_{j,1}^i+\sum_{i,j}u_{j,2}^{(i)}\zeta_{j,2}^i+f,\psi_k^{(l)}\right)~\forall l,k,\\
    \left(\sum_{i,j}(u_{j,1}^{(i)})_{tt}\phi_{glo,1,j}^{i}+\sum_{i,j}(u_{j,2}^{(i)})_{tt}\phi_{glo,2,j}^{i},\xi_k^{(l)}\right)=\left(\sum_{i,j}u_{j,1}^{(i)}\zeta_{j,1}^i+\sum_{i,j}u_{j,2}^{(i)}\zeta_{j,2}^i+f,\xi_k^{(l)}\right)~\forall l,k.
\end{aligned}
\end{equation}
Since
$$\Big(\phi_{glo,2,j}^{i},\psi_k^{(l)}\Big)=0,~\Big(\phi_{glo,1,j}^{i},\psi_k^{(l)}\Big)=\delta_{il}\delta_{jk},~\Big(\phi_{glo,2,j}^{i},\xi_k^{(l)}\Big)=\delta_{il}\delta_{jk},~\Big(\phi_{glo,1,j}^{i},\xi_k^{(l)}\Big)=0,$$
(\ref{ml1}) can be rewritten as 
\begin{equation}\label{ml2}
\begin{aligned}
    \Big(u_{k,1}^{(l)}\Big)_{tt}=\left(\sum_{i,j}u_{j,1}^{(i)}\zeta_{j,1}^i+\sum_{i,j}u_{j,2}^{(i)}\zeta_{j,2}^i+f,\psi_k^{(l)}\right)~\forall l,k,\\
\Big(u_{k,2}^{(l)}\Big)_{tt}=\left(\sum_{i,j}u_{j,1}^{(i)}\zeta_{j,1}^i+\sum_{i,j}u_{j,2}^{(i)}\zeta_{j,2}^i+f,\xi_k^{(l)}\right)~\forall l,k.
\end{aligned}
\end{equation}
Since 
\begin{equation}\label{ml3}
\begin{aligned}
    \left(\sum_{i,j}u_{j,1}^{(i)}\zeta_{j,1}^i+\sum_{i,j}u_{j,2}^{(i)}\zeta_{j,2}^i,v\right)=-a(u_H,v),
\end{aligned}
\end{equation}
and consider
\begin{equation}\label{ml4}
\begin{aligned}
    \left(\sum_{i,j}u_{j,1}^{(i)}\zeta_{j,1}^i+\sum_{i,j}u_{j,2}^{(i)}\zeta_{j,2}^i,\psi_k^{(l)}\right)=\left(\sum_{i,j}u_{j,1}^{(i)}\zeta_{j,1}^i+\sum_{i,j}u_{j,2}^{(i)}\zeta_{j,2}^i,\phi_{glo,1,k}^{l}\right)=-a(u_H,\phi_{glo,1,k}^{l}),\\
    \left(\sum_{i,j}u_{j,1}^{(i)}\zeta_{j,1}^i+\sum_{i,j}u_{j,2}^{(i)}\zeta_{j,2}^i,\xi_k^{(l)}\right)=\left(\sum_{i,j}u_{j,1}^{(i)}\zeta_{j,1}^i+\sum_{i,j}u_{j,2}^{(i)}\zeta_{j,2}^i,\phi_{glo,2,k}^{l}\right)=-a(u_H,\phi_{glo,2,k}^{l}),
\end{aligned}
\end{equation}
we have
$$\Big(u_{k,1}^{(l)}\Big)_{tt}+\sum_{i,j}u_{j,1}^{(i)}a\Big(\phi_{j,1}^{(i)},\phi_{k,1}^{(l)}\Big)+\sum_{i,j}u_{j,2}^{(i)}a\Big(\phi_{j,2}^{(i)},\phi_{k,1}^{(l)}\Big)=\Big(f,\psi_k^{(l)}\Big) ~\forall l,k,$$
$$\Big(u_{k,2}^{(l)}\Big)_{tt}+\sum_{i,j}u_{j,1}^{(i)}a\Big(\phi_{j,1}^{(i)},\phi_{k,2}^{(l)}\Big)+\sum_{i,j}u_{j,2}^{(i)}a\Big(\phi_{j,2}^{(i)},\phi_{k,2}^{(l)}\Big)=\Big(f,\xi_k^{(l)}\Big) ~\forall l,k.$$
In this case, we can obtain a diagonal mass matrix. The mass lumping scheme, which approximates the original mass matrix $M_{ij}=(\phi_{glo,i},\phi_{glo,j})$ for $\phi_{glo,i},~\forall\phi_{glo,j}\in V_{glo}$ with a diagonal mass matrix $\tilde{M}=(\pi(\phi_{glo,i}),\pi(\phi_{glo,j}))$. This approximation allows us to simplify the computations significantly. By the construction of the auxiliary basis functions, we have 
\begin{equation}\label{masslumping}
\begin{aligned}
      (\phi_{glo,i},\phi_{glo,j})=&\big(\pi(\phi_{glo,i}),\pi(\phi_{glo,j})\big)+\big((I-\pi)(\phi_{glo,i}),(I-\pi)(\phi_{glo,j})\big)\\\approx&\big(\pi(\phi_{glo,i}),\pi(\phi_{glo,j})\big)+O(H^2).
\end{aligned}
\end{equation}



As discussed before, the system \eqref{0.4} is still coupled through the first two terms. After applying the mass lumping scheme, it can be approximated as: for $n\geq1$, find $u_{H,1}\in V_{H,1}$ and $u_{H,2}\in V_{H,2}$ such that
\begin{equation}\label{pis}
\begin{aligned}
     &b\left(\frac{u_{H,1}^{n+1}-2u_{H,1}^{n}+u_{H,1}^{n-1}}{\tau^2},w\right)+\frac{1}{2}a\left(u_{H,1}^{n+1}+u_{H,1}^{n-1}+2u_{H,2}^{n},w\right)=b\left(f^n,w\right)~ \forall w \in V_{glo,1},\\
     &b\left(\frac{u_{H,2}^{n+1}-2u_{H,2}^{n}+u_{H,2}^{n-1}}{\tau^2},w\right)+a\left(\omega u_{H,1}^{n}+(1-\omega)\left(\frac{u_{H,1}^{n+1}+u_{H,1}^{n-1}}{2}\right)+u_{H,2}^{n},w\right)=b\left(f^n,w\right)~\\&~~~~~~~\forall w \in V_{glo,2},
\end{aligned}
\end{equation}
where the bilinear form $b(v,w)$ is defined as
$$b(v,w)=(\pi(v),\pi(w)),$$ and define the $b$ norm as $\|u\|^2_b=b(u,u).$
\begin{remark}
    In our numerical examples, we use fully implicit scheme for the comparison, which is shown as follows
\begin{equation}\label{fully implicit}
\begin{aligned}
     &b\left(\frac{u_{H}^{n+1}-2u_{H}^{n}+u_{H}^{n-1}}{\tau^2},w\right)+\frac{1}{2}a\left(u_{H}^{n+1}+u_{H}^{n-1},w\right)=b\left(f^n,w\right)~ \forall w \in V_{glo}.
\end{aligned}
\end{equation}
    
\end{remark}

\subsection{The construction of the local spaces}
Based on the analysis in \cite{chung2018constraint}, the global basis functions exhibit an exponential decay property, meaning that their values become very small at locations far from the block $K_i$. Consequently, in this subsection, we can localize the global basis functions by constructing local basis functions on a suitably enlarged oversampling domain to achieve a better approximation. More precisely, the oversampling domain is denoted as $K_{i,m}$, which is formed by enlarging the coarse grid element $K_i$ by $m$ coarse grid layers.

The multiscale basis functions $\phi_{j,1}^{(i)}$ of $V_{H,1}$ are obtained by finding $\Big(\phi_{j,1}^{(i)},\mu_{j,1}^{(i)}\Big)\in V_{H,1}(K_{i,m})\times (V_{\text{aux},1}\oplus 
 V_{\text{aux},2})$, such that
\begin{equation}\label{eq0}
\begin{aligned}
a\Big(\phi_{j,1}^{(i)},v\Big)+\Big(\mu_{j,1}^{(i)},v\Big)&=0~\forall v\in V(K_{i,m}),\\
\Big(\phi_{j,1}^{(i)}-\psi_k^{(l)},\chi_1\Big)&=0,~\forall\chi_1\in V_{\text{aux,1}}\\
\Big(\phi_{j,1}^{(i)},\xi_k^{(l)}\Big)&=0.
\end{aligned}
\end{equation}
Similarly, the multiscale basis functions $\phi_{j,2}^{(i)}$ of $V_{H,2}$ are obtained by finding 
$\Big(\phi_{j,2}^{(i)},\mu_{j,2}^{(i)}\Big)\in V_{H,2}(K_{i,m})\times (V_{\text{aux},1}\oplus V_{\text{aux},2})$, such that
\begin{equation}\label{eq1}
\begin{aligned}
    a\Big(\phi_{j,2}^{(i)},v\Big)+\Big(\mu_{j,2}^{(i)},v\Big)&=0~\forall v\in V(K_{i,m}),\\
    \Big(\phi_{j,2}^{(i)},\psi_k^{(l)}\Big)&=0,\\
    \Big(\phi_{j,2}^{(i)}-\xi_k^{(l)},\chi_2\Big)&=0~\forall \chi_2\in V_{\text{aux,2}}.
\end{aligned}
\end{equation}
Then, the multiscale finite element spaces $V_{H,1}$ and $V_{H,2}$ are defined as
$$V_{H,1}=\text{span}_{i,j}\Big\{\phi_{j,1}^{(i)}:1\leq i\leq N,1\leq j\leq l_i\Big\},\;V_{H,2}=\text{span}_{i,j}\Big\{\phi_{j,2}^{(i)}:1\leq i\leq N,1\leq j\leq L_i\Big\}.$$

\section{Stability and convergence analysis}\label{004}
In this section, we analyze the stability and convergence of the proposed localized coarse-grid model (\ref{0.4}). 
\subsection{Stability analysis}
In this subsection, the stability of the proposed scheme as $\omega=1$ will be discussed as the case $\omega=0$ performs similarly. To simplify the discussion, we consider $f=0$.

We seek the solution $u_{H}=u_{H,1}+u_{H,2}$ with $\{u_{H,1}^{n}\}_{n=1}^N \in V_{H,1}$, $\{u_{H,2}^{n}\}_{n=1}^N \in V_{H,2}$, such that
\begin{equation}\label{1.1}
    b\left(u_{H}^{n+1}-2u_{H}^{n}+u_{H}^{n-1},w\right)+\frac{\tau^2}{2}a\left(u_{H,1}^{n+1}+u_{H,1}^{n-1}+2u_{H,2}^{n},w\right)=0~ \forall w \in V_{H,1},
\end{equation}
\begin{equation}\label{1.2}
    b\left(u_{H}^{n+1}-2u_{H}^{n}+u_{H}^{n-1},w\right)+\tau^2 a\left(u_{H,1}^{n}+u_{H,2}^{n},w\right)=0~\forall w \in V_{H,2}.
\end{equation}
We define the discrete energy $E^{n+\frac{1}{2}}$ of $u_{H}$ as
\begin{equation}\label{1.3}
\begin{aligned}
     E^{n+\frac{1}{2}}=&\|u_{H}^{n+1}-u_{H}^{n}\|_{b}^2+\frac{\tau^2}{2}\sum_{i=1,2}\left(\|u_{H,i}^{n+1}\|_{a}^2+\|u_{H,i}^{n}\|_{a}^2\right)\\
     &+\tau^2 a(u_{H,2}^{n+1},u_{H,1}^{n})+\tau^2 a(u_{H,1}^{n+1},u_{H,2}^{n})-\frac{\tau^2}{2}\|u_{H,2}^{n+1}-u_{H,2}^{n}\|_{a}^2.
\end{aligned}
\end{equation}

\begin{lemma}\label{lemma1}
    For $u_{H,1}$ and $u_{H,2}$ satisfying $(\ref{1.1})$ and $(\ref{1.2})$, we have $$E^{n+\frac{1}{2}}=E^{n-\frac{1}{2}}.$$ 
\end{lemma}
The proof of Lemma \ref{lemma1} is provided in \ref{appl1}.

\begin{theorem}
    Consider the scheme $(\ref{1.1})$ and $(\ref{1.2})$. If \begin{equation}\label{eq:CFL}
        \|v_{2}\|_b^2 \geqslant \frac{\tau^2}{2}\|v_{2}\|_{a}^2, \; \forall v_{2}\in V_{H,2},
    \end{equation} the partially explicit scheme is stable with $$\Big(\|u_{H,1}^{n}\|_a^2+\|u_{H,2}^{n+1}\|_a^2+\|u_{H,2}^{n}\|_a^2+\|u_{H,1}^{n+1}\|_a^2\Big)\leq \cfrac{2}{\tau^2(1-\gamma)}E^{\frac{1}{2}},\; \forall n\geq 1,$$
    where $\gamma_a = \sup_{u_2\in V_{H,2},u_1\in V_{H,1}}\cfrac{a(u_1,u_2)}{\|u_1\|_a\|u_2\|_a}<1$.
\end{theorem}
\begin{proof}
    Using \eqref{eq:CFL}, we can show that
\begin{equation}\label{1.4}
\begin{aligned}
     E^{n+\frac{1}{2}}=&\|u_{H}^{n+1}-u_{H}^{n}\|_{b}^2+\frac{\tau^2}{2}\sum_{i=1,2}\left(\|u_{H,i}^{n+1}\|_{a}^2+\|u_{H,i}^{n}\|_{a}^2\right)\\
     &+\tau^2 a(u_{H,2}^{n+1},u_{H,1}^{n})+\tau^2 a(u_{H,1}^{n+1},u_{H,2}^{n})-\frac{\tau^2}{2}\|u_{H,2}^{n+1}-u_{H,2}^{n}\|_{a}^2\\
     \geqslant&\|u_{H,1}^{n+1}-u_{H,1}^{n}\|_{b}^2+\frac{\tau^2}{2}\sum_{i=1,2}\left(\|u_{H,i}^{n+1}\|_{a}^2+\|u_{H,i}^{n}\|_{a}^2\right)\\
     &+\tau^2 a(u_{H,2}^{n+1},u_{H,1}^{n})+\tau^2 a(u_{H,1}^{n+1},u_{H,2}^{n}).
\end{aligned}
\end{equation}
Since $V=V_{H,1}\oplus V_{H,2}$ with $|a(u_1,u_2)|\leq \gamma_a\|u_1\|_a\|u_2\|_a$ for all $u_1\in V_{H,1}\setminus \{0\}$, $u_2\in V_{H,2}\setminus\{0\} $, we obtain
\begin{equation}\label{eq:stablity}
    |a(u_{H,1}^{n},u_{H,2}^{n+1})+a(u_{H,2}^{n},u_{H,1}^{n+1})|\leq \frac{\gamma_a}{2}\Big(\|u_{H,1}^{n}\|_a^2+\|u_{H,2}^{n+1}\|_a^2+\|u_{H,2}^{n}\|_a^2+\|u_{H,1}^{n+1}\|_a^2\Big).
\end{equation}
Thus, $E^{n+\frac{1}{2}}$ defines a norm. Combining \eqref{1.4} and \eqref{eq:stablity}, we obtain
$$ \frac{\tau^2(1-\gamma_a)}{2}\Big(\|u_{H,1}^{n}\|_a^2+\|u_{H,2}^{n+1}\|_a^2+\|u_{H,2}^{n}\|_a^2+\|u_{H,1}^{n+1}\|_a^2\Big)\leq E^{n+\frac{1}{2}},$$
and using Lemma \ref{lemma1},
$$\Big(\|u_{H,1}^{n}\|_a^2+\|u_{H,2}^{n+1}\|_a^2+\|u_{H,2}^{n}\|_a^2+\|u_{H,1}^{n+1}\|_a^2\Big)\leq  \cfrac{2}{\tau^2(1-\gamma)}E^{n+\frac{1}{2}}=\cfrac{2}{\tau^2(1-\gamma)}E^{\frac{1}{2}}.$$
\end{proof}    

\subsection{Convergence analysis}
Before proceeding with our convergence analysis, we need to define two operators. The first one is a solution map $G:~L^2(\Omega)\rightarrow{V:=H^1_0(\Omega)}$ defined by the following: for any $g\in L^2(\Omega)$, the image $Gg\in V$ is defined as
$$a(Gg,w)=(g,w)_{L^2(\Omega)}~\forall w\in V.$$
Then, define an elliptic operator $P_H$: $V\rightarrow{V_{H,1}^{(m)}+V_{H,2}^{(m)}}$ by the following: for any $v\in V$, the image $P_Hv\in V_{H,1}^{(m)}+V_{H,2}^{(m)}$ is defined as 
$$a(P_Hv,w)=a(v,w)~\forall w\in V_{H,1}^{(m)}+V_{H,2}^{(m)}.$$ 
The approximation error of the projection depends on the exponential decaying property of the eigenvalues of local spectral problems.
\begin{lemma}\label{cem}
\text{(See \cite{chung2018constraint})}
    First of all, for any $v\in V_h$, it has $$\|(I-\pi)(v)\|_{L^2(\Omega)}^2\leq2\Lambda^{-1}H^2a(v,v).$$ Moreover, there exists $\beta>0$ depending on the regularity of the coarse grid and the fine grid and the eigenvalue decay, and independent on the mesh size $H$,  such that the following inverse inequality holds
\begin{equation}\label{L.cem1}
\begin{aligned}
     \|v\|_b^2\geq \beta\kappa_1^{-1}H^2a(v,v),~\forall v\in V_H^{(m)}.
\end{aligned}
\end{equation}
\end{lemma}
As a direct consequence of Lemma \ref{cem}, for any $v\in V_h$, it can be derived that
\begin{equation}\label{L.cem2}
\begin{aligned}
     \|v\|_{L^2(\Omega)}^2\geq \left(1+\frac{2\kappa_1}{\beta\Lambda}\right)\|v\|^2_{b}.
\end{aligned}
\end{equation}

\begin{lemma}
    Suppose the size of oversampling region $m$ satisfies the following condition as the coarse grid $H$ refines $m=O\left(log\left(\frac{\kappa_1}{H}\right)\right).$ Then, there holds
\begin{equation}\label{L.1}
\begin{aligned}
     \|(I-P_H)Gg\|_a\leq C\Lambda^{-\frac{1}{2}}H\|g\|_{L^2(\Omega)}~\forall g\in L^2(\Omega),
\end{aligned}
\end{equation}
and
\begin{equation}\label{L.2}
\begin{aligned}
     \|(I-P_H)Gg\|_{L^2(\Omega)}\leq C\Lambda^{-1}H^2\|g\|_{L^2(\Omega)}~\forall g\in L^2(\Omega).
\end{aligned}
\end{equation}
\end{lemma}

Now, we aim to estimate the error between the solution $u^{n}$ and the coarse-scale solutions $u_{H}^{n}$. We define the quantities
\begin{equation}\label{2.1}
\begin{aligned}
     \epsilon^{n} = u^{n}-\sum_{i=1,2} u_{H,i}^{n},~
     \delta^{n}_i = u_{H,i}^{n}-P_{H,i}u^{n},~\delta^{n}=\sum_{i=1,2}\delta^{n}_i,~
     \eta^{n}_i = (I-P_{H,i})u^{n},~\eta^{n}=\sum_{i=1,2}\eta^{n}_i.
\end{aligned}
\end{equation}

\begin{theorem}\label{2.01}
Assuming $f\in C^4\big([0,T];L^2(\Omega))\cap C([0,T];H^1(\Omega)\big)$ and $u\in C^4\big([0,T];L^2(\Omega))\cap C^2([0,T];H(\Omega)\big),$ we have the following estimates
\begin{equation}\label{2.2}
\begin{aligned}
     \tau (R_1^n+R_3^n)+\Lambda^{-\frac{1}{2}}HR_2^n\leq C(\Lambda^{-1}H^2+\tau^2),
\end{aligned}
\end{equation}
where 
\begin{equation}\label{2.03}
\begin{aligned}
     R_1^n=&\sum_{k=1}^n\left\|\pi\left(\frac{\eta^{n+1}-2\eta^n+\eta^{n-1}}{\tau^2}\right)\right\|_{L^2(\Omega)}+ \sum_{k=1}^n\left\| (I-\pi)\frac{\partial^2 u}{\partial t^2}(\cdot,t_n)\right\|_{L^2(\Omega)}\\
     &+\sum_{k=1}^n\left\|\pi\left(\frac{\partial^2u}{\partial t^2}(\cdot,t_n)-\frac{u^{n+1}-2u^n+u^{n-1}}{\tau^2}\right)\right\|_{L^2(\Omega)},\\
     R_2^n=&\tau\sum_{k=1}^{n-1}\left\|(I-\pi)(\frac{f^{k+1}-f^k}{\tau})\right\|_{L^2(\Omega)}+\|(I-\pi)(f^1)\|_{L^2(\Omega)}\\
     &+\|(I-\pi)(f^n)\|_{L^2(\Omega)},\\
     R_3^n=&\sum_{k=1}^n\left\|\left(\frac{P_{H,1}u^{n+1}-2P_{H,1}u^n+P_{H,1}u^{n-1}}{2}\right)\right\|_a.
\end{aligned}
\end{equation}
\end{theorem}
The proof of Theorem \ref{2.01} is provided in \ref{appb}.

\begin{theorem}\label{2.02}
Assume $f\in C^4\big([0,T];L^2(\Omega))\cap C([0,T];H^1(\Omega)\big)$ and $u\in C^4\big([0,T];L^2(\Omega))\cap C^2([0,T];H(\Omega)\big),$ we have the following estimates
\begin{equation}\label{2.4}
\begin{aligned}
     \max_{0\leq k\leq N_{T}-1}\left\|\frac{\epsilon^{n+1}+\epsilon^k}{2}\right\|_{L^2(\Omega)}\leq C(\Lambda^{-1}H^2+\tau^2).
\end{aligned}
\end{equation}
\end{theorem}
The proof of Theorem \ref{2.02} is provided in \ref{appd}.

\section{The implementation of partially explicit Runge-Kutta method for wave equations}\label{005}
In this section, we develop the partially explicit Runge-Kutta scheme for the system (\ref{0.21}).
Denote $r_H=(u_H)_t$, (\ref{0.21}) can be rewritten as
\begin{equation}\label{rk1}
\left\{
\begin{aligned}
\left(\frac{\partial u_H}{\partial t},v\right) &= (r_H,v),\\
\left(\frac{\partial r_H}{\partial t} ,v\right)&= -a(u_H,w)+(f,w).
\end{aligned}
\right.
\end{equation}

Denote $M$ and $A$ as the matrices of $L^2$ inner product $(u,v)_{L^2(\Omega)}$ and bilinear form $a(u,v)$ with respect to the coarse grid basis functions in $V_H$. Let $U^n,~R^n$ and $F^n$ be the $n$-th time step column vector consisting of the coordinate representation of $u_H^n,r_H^n\in V_{H,1}\oplus V_{H,2}$, and the source term $f$ with respect to the basis functions $\phi_{j,1}^{(i)}$ and $\phi_{j,2}^{(i)}$. Then (\ref{rk1}) can be rewritten as the matrix form
\begin{equation}\label{rk31}
\begin{aligned}
\left(\begin{matrix}
    M&0\\0&M
\end{matrix}\right)\frac{\partial}{\partial t}\left(\begin{matrix}
    U^n\\R^n
\end{matrix}\right)=\left(\begin{matrix}
    0&M\\-A&0
\end{matrix}\right)\left(\begin{matrix}
    U^n\\R^n
\end{matrix}\right)+\left(\begin{matrix}
    0\\MF^n
\end{matrix}\right).
\end{aligned}
\end{equation}
 Thus, we have 
\begin{equation}\label{rk4}
\begin{aligned}
\frac{\partial}{\partial t}\left(\begin{matrix}
    U^n\\R^n
\end{matrix}\right)=\left(\begin{matrix}
    0&I\\-M^{-1}A&0
\end{matrix}\right)\left(\begin{matrix}
    U^n\\R^n
\end{matrix}\right)+\left(\begin{matrix}
    0\\F^n
\end{matrix}\right),
\end{aligned}
\end{equation}
where $I$ is the identity matrix. 
After spatial splitting, let $M_{ij}=(u_i,v_j)$ and $A_{ij}=a(u_i,v_j),~1\leq i,j\leq2$, where $u_i\in V_{H,i}$, $v_j\in V_{H,j}$. Denote $U_i^n$ and $R_i^n,~i=1,2$ are the $n$-th time step column vectors consisting of the coordinate representation of the basis functions in $V_{H,1}$ and  $V_{H,2}$. $F^n$ denotes as the $n$-th time step column vector consisting of the coordinate representation of the source term $f$.
By using the mass lumping scheme \eqref{masslumping}, we can respectively choose
\begin{equation}\label{rk5}
\begin{aligned}
y_1\left(\begin{matrix}
    U^n_1\\R^n_1
\end{matrix}\right)=\left(\begin{matrix}
    0&I\\-M_{11}^{-1}A_{11}&0
\end{matrix}\right)\left(\begin{matrix}
    U^n_1\\R^n_1
\end{matrix}\right)+\left(\begin{matrix}
    0&M_{11}^{-1}M_{12}\\-M_{11}^{-1}A_{12}&0
\end{matrix}\right)\left(\begin{matrix}
    U^n_2\\R^n_2
\end{matrix}\right)+\left(\begin{matrix}
    0\\F_1^n
\end{matrix}\right)
\end{aligned}
\end{equation}
in $V_{H,1}$ (implicit), and
\begin{equation}\label{rk7}
\begin{aligned}
y_2\left(\begin{matrix}
    U^n_2\\R^n_2
\end{matrix}\right)=\left(\begin{matrix}
    0&M_{22}^{-1}M_{21}\\-M_{22}^{-1}A_{21}&0
\end{matrix}\right)\left(\begin{matrix}
    U^n_1\\R^n_1
\end{matrix}\right)+\left(\begin{matrix}
    0&I\\-M_{22}^{-1}A_{22}&0
\end{matrix}\right)\left(\begin{matrix}
    U^n_2\\R^n_2
\end{matrix}\right)+\left(\begin{matrix}
    0\\F_2^n
\end{matrix}\right)
\end{aligned}
\end{equation}
in $V_{H,2}$ (explicit). Here, $M_{11},M_{22}$ are both diagonal matrices, while $M_{21}$ and $M_{12}$ are zero. 
Thus, there exist constants $C_1=M_{11}^{-1}$ and $C_2=M_{22}^{-1}$ such that
\begin{equation}\label{rkh1}
\begin{aligned}
y_1\left(\begin{matrix}
    U^n_1\\R^n_1
\end{matrix}\right)=C_1\left(\begin{matrix}
    0&I\\-A_{11}&0
\end{matrix}\right)\left(\begin{matrix}
    U^n_1\\R^n_1
\end{matrix}\right)+C_1\left(\begin{matrix}
    0&0\\-A_{12}&0
\end{matrix}\right)\left(\begin{matrix}
    U^n_2\\R^n_2
\end{matrix}\right)+\left(\begin{matrix}
    0\\F_1^n
\end{matrix}\right),
\end{aligned}
\end{equation}
and
\begin{equation}\label{rkh2}
\begin{aligned}
y_2\left(\begin{matrix}
    U^n_2\\R^n_2
\end{matrix}\right)=C_2\left(\begin{matrix}
    0&0\\-A_{21}&0
\end{matrix}\right)\left(\begin{matrix}
    U^n_1\\R^n_1
\end{matrix}\right)+C_2\left(\begin{matrix}
    0&I\\-A_{22}&0
\end{matrix}\right)\left(\begin{matrix}
    U^n_2\\R^n_2
\end{matrix}\right)+\left(\begin{matrix}
    0\\F_2^n
\end{matrix}\right).
\end{aligned}
\end{equation}

The Butcher tableau for the Runge-Kutta method can be found in \cite{ascher1997implicit}. In our numerical examples, we employ a four-stage, third order partially explicit Runge-Kutta scheme, which is detailed in the Table \ref{RK4}. For comparison, we also utilize a fully implicit Runge-Kutta scheme, presented in the left portion of Table \ref{RK4}. This allows us to evaluate the performance differences between the fully implicit and partially explicit schemes.
\begin{table}[H]
    \begin{minipage}{0.49\textwidth}
        \bgroup
        \renewcommand{\arraystretch}{1.5}
        $$
        \begin{array}{c|cccc}
        \frac{1}{2} & \frac{1}{2} & 0 & 0 & 0 \\
        \frac{2}{3} & \frac{1}{6} & \frac{1}{2} & 0 & 0 \\
        \frac{1}{2} & -\frac{1}{2} & \frac{1}{2} & \frac{1}{2} & 0 \\
        1 & \frac{3}{2} & -\frac{3}{2} & \frac{1}{2} & \frac{1}{2} \\
        \hline
         & \frac{3}{2} & -\frac{3}{2} & \frac{1}{2} & \frac{1}{2}
        \end{array}
        $$
        \egroup
    \end{minipage}
    \begin{minipage}{0.49\textwidth}
        \bgroup
        \renewcommand{\arraystretch}{1.5} 
        $$
        \begin{array}{c|ccccc}
        0 & 0 & 0 & 0 & 0 & 0 \\
        \frac{1}{2} & \frac{1}{2} & 0 & 0 & 0 & 0 \\
        \frac{2}{3} & \frac{11}{18} & \frac{1}{18} & 0 & 0 & 0 \\
        \frac{1}{2} & \frac{5}{6} & -\frac{5}{6} & \frac{1}{2} & 0 & 0 \\
        1 & \frac{1}{4} & \frac{7}{4} & \frac{3}{4} & -\frac{7}{4} & 0 \\
        \hline
         & \frac{1}{4} & \frac{7}{4} & \frac{3}{4} & -\frac{7}{4} & 0
        \end{array}
        $$
        \egroup
    \end{minipage}
    \caption{The Butcher tableau of a four-stage, third order partially explicit Runge-Kutta scheme.}
    \label{RK4}
\end{table}

Then, one step from $t_{n}$ to $t_{n+1}=t_n+\tau$ of the partially explicit scheme is given as follows

Set $\tilde{K}_1=y_2(u_n)$.

For $i=1,\dots,s$ do:

$\bullet$ Solve for $K_i$:

$K_i=y_1(u_i),$ where $u_i=u_n+\tau\sum_{j=1}^{i}a_{i,j}K_j+\tau\sum_{j=1}^{i}\tilde{a}_{i,j}\tilde{K}_j$.

$\bullet$ Evaluate

$\tilde{K}_{i+1}=y_2(u_i)$.

Finally, evaluate $u_{n+1}=u_n+\tau\sum_{j=1}^sb_jK_j+\tau\sum_{j=1}^{\sigma}\tilde{b}_j\tilde{K}_j.$

\section{Numerical examples}\label{006}
 In this section, we present numerical examples with three different high-contrast media and two source terms to demonstrate the convergence of our proposed method (the source terms are shown in Figure \ref{fig:C3}).
We take the spatial domain $\Omega=[0,1]^2$. To enable a comparison of convergence rates, we first introduce some necessary notions.
The following error quantities are used to compare the convergence of our algorithm:
\begin{equation}\label{n1}
\begin{aligned}
e_2&=\frac{\|u_h-\sum_{i=1,2}u_{H,i}\|_{L^2(\Omega)}}{\|u_h\|_{L^2(\Omega)}},~e_a=\frac{\|u_h-\sum_{i=1,2}u_{H,i}\|_{a}}{\|u_h\|_{a}},\\
e_b&=\frac{\|u_h-\sum_{i=1,2}u_{H,i}\|_{b}}{\|u_h\|_{b}}.
\end{aligned}
\end{equation}
To obtain the convergence rate of the central difference scheme and the Runge-Kutta scheme, we use the time step sizes $\tau_l= 2^{-l}\times5\times10^{-3}$ for $0\leq l\leq6$. The reference solution is chosen as the solution computed with the smallest time step size, which is $\tau=2^{-6}\times5\times10^{-3}$. Let $u^{c}_{l}$ be the final solution obtained using the central difference scheme with the $l$-th time step size $\tau_l$, and let $u^{c}_{s}$ be the solution computed with the smallest time step size  using central difference scheme. Then the central difference temporal error $e_c$ can be defined as follows
\begin{equation}\label{err2}
\begin{aligned}
e^{c}_{l}&=\frac{\|u^{c}_{l}-u^{c}_{s}\|_{L^2(\Omega)}}{\|u^{c}_{s}\|_{L^2(\Omega)}}.
\end{aligned}
\end{equation}
Similarly, let $u^{r}_{l}$ be the final solution obtained with the $l$-th time step $\tau_l$ using Runge-Kutta scheme, and let $u^{r}_{s}$ be the error with the smallest time step using Runge-Kutta scheme. Then the Runge-Kutta temporal error $e^{r}_l$ can be defined as follows
\begin{equation}\label{err22}
\begin{aligned}
e^{r}_l&=\frac{\|u^{r}_{l}-u^{r}_{s}\|_{L^2(\Omega)}}{\|u^{r}_{s}\|_{L^2(\Omega)}}.
\end{aligned}
\end{equation}


 \begin{figure}[H]
    \centering
    \begin{subfigure}{0.42\textwidth}
        \includegraphics[width=\linewidth]{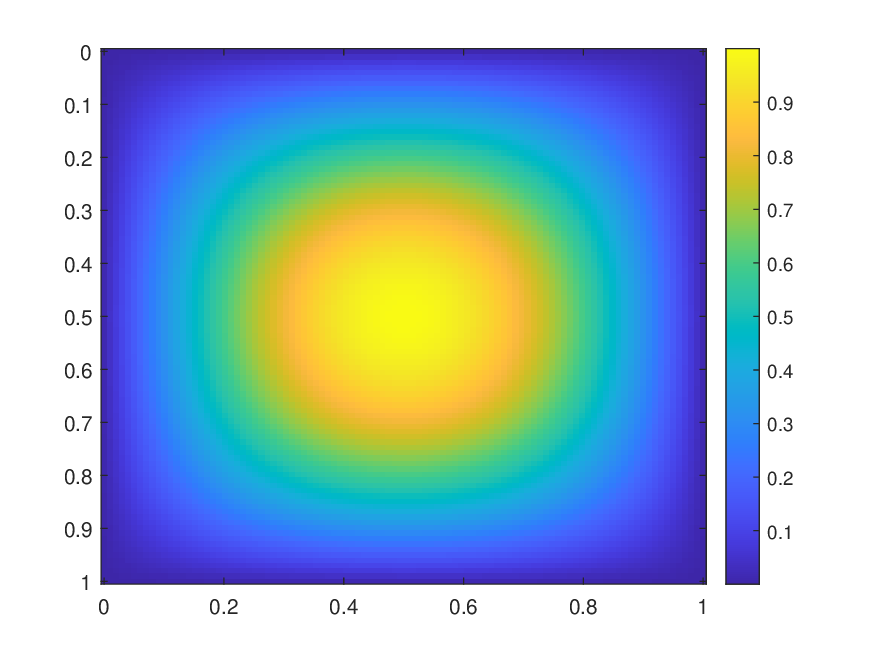}
        \caption{}
        \label{fig:C31}
    \end{subfigure}
    \begin{subfigure}{0.42\textwidth}
        \includegraphics[width=\linewidth]{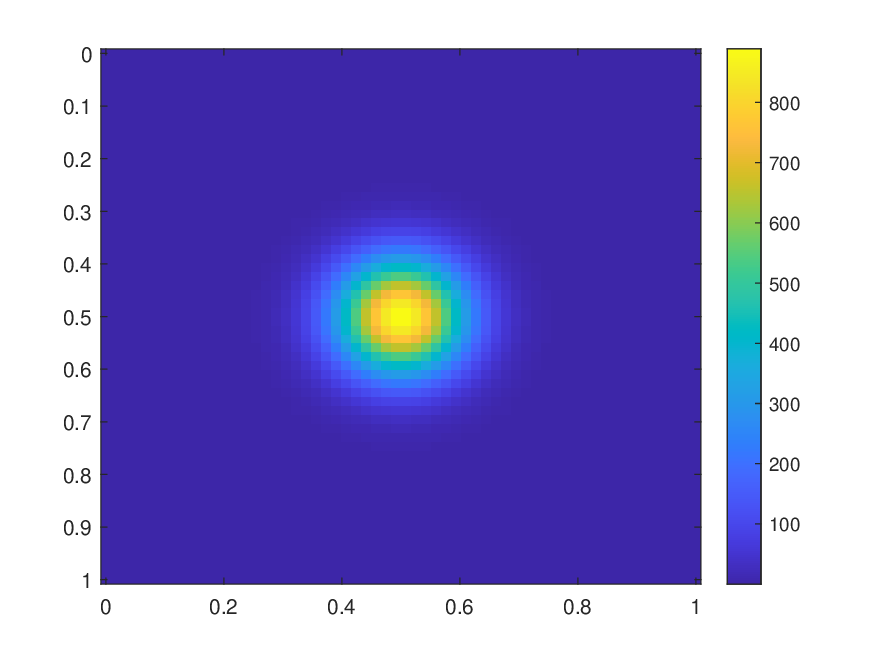}
        \caption{}
        \label{fig:C32}
    \end{subfigure}
    \caption{Left: the source term $f$. Right: the source term $f_2$.}
    \label{fig:C3}
\end{figure}

\subsection{Case 1}
In our first numerical example, the source function $f$ is considered as 
\begin{equation}\label{f21}
\begin{aligned}
f(t,x,y)=\sin(z_0t)\sin(\pi x)\sin(\pi y),
\end{aligned}
\end{equation}
for all $(t,x,y)\in \Omega$, where the fine grid parameter and central frequency are chosen as $h=1/100$ and $z_0=300$ (Figure \ref{fig:C3}).
For this case, we set $\kappa_{\text{cutoff}}=1$ and choose the coarse grid size as $H=1/10$. Using the fully discrete scheme, we solve for the numerical solution at the final time $T=0.4$, with the time step size $\tau=2.5\times10^{-3}$ (see Figure \ref{fig:k1}). We plot the reference solution computed on the fine grid with the fully implicit scheme (upper right), the solution computed by CEM-GMsFEM with the partially explicit scheme via mass lumping using the central difference scheme (lower left) and the Runge-Kutta scheme (lower right). 
From Table \ref{tablek} and Table \ref{tableMKc300}, we compare the performance of different choices of contrast values $\kappa$ and implicit and partially explicit time discretization schemes respectively. It shows that for different contrast values, errors remain consistent in the $L^2$ norm, energy norm and $b$-norm, indicating that the convergence rate is independent of the contrast (since the basis functions representing the channels are included in the $V_{\text{aux,1}}$ space). 
By employing \eqref{fully implicit} for the fully implicit central difference scheme, and using \eqref{rkh1}, \eqref{rkh2}, along with the left portion of Table \ref{RK4} for the fully implicit Runge-Kutta scheme, both the fully implicit and partially explicit schemes provide similar results.
Additionally, using \eqref{fully implicit}, \eqref{rkh1}, \eqref{rkh2}, and Table \ref{RK4}, we find that both the fully implicit and partially explicit schemes yield similar results. However, for the fully explicit scheme, stability is only achieved until the time step size is reduced to $7.8125\times 10^{-5}$, which significantly increases computational cost.
It is also observed that after applying mass lumping, the results closely match the original solution (see Table \ref{tableMKc30} and Figure \ref{fig:mll}). From Table \ref{tableKPP20}, Table \ref{tablec0} and Figure \ref{fig:cc21}, it is evident that the second order central difference scheme converges more slowly than the third order Runge-Kutta scheme, and the average convergence rates align well with theoretical predictions.
\begin{figure}[H]
    \centering
    \begin{subfigure}[b]{0.37\textwidth}
        \includegraphics[width=\linewidth]{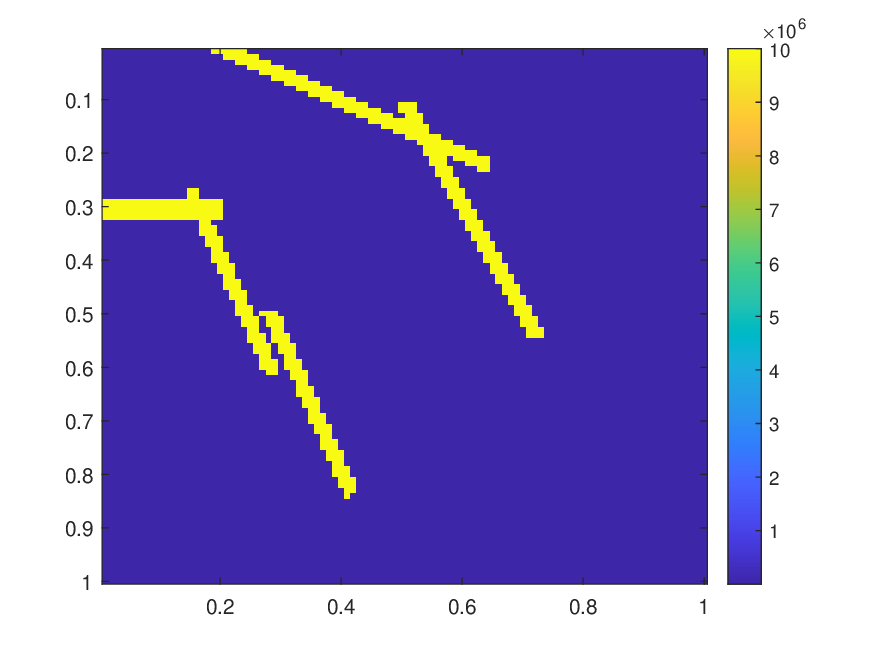}
        \caption{}
        \label{fig:image011}
    \end{subfigure}
    \begin{subfigure}[b]{0.37\textwidth}
        \includegraphics[width=\linewidth]{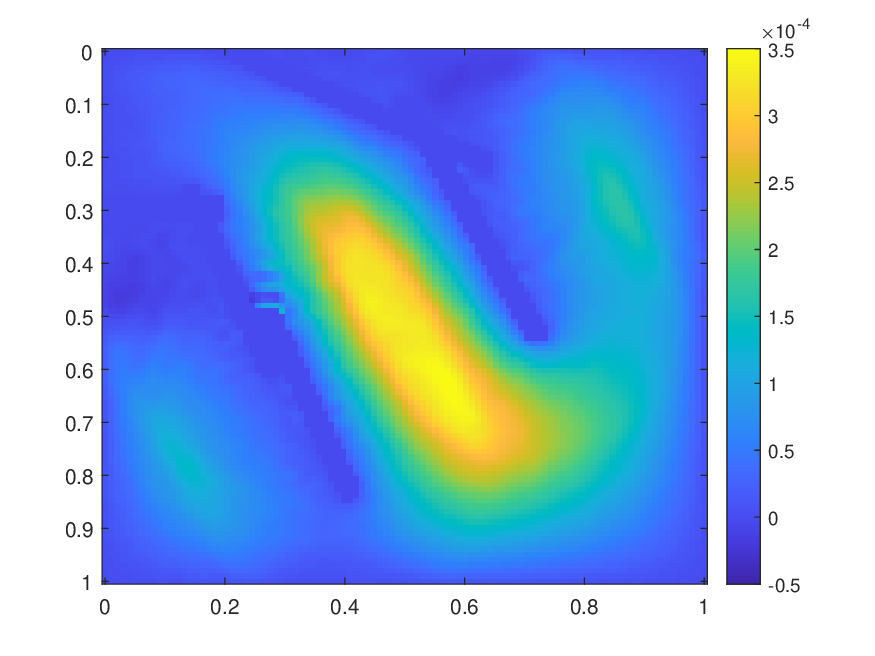}
        \caption{}
        \label{fig:image021}
    \end{subfigure}

    \begin{subfigure}[b]{0.37\textwidth}
        \includegraphics[width=\linewidth]{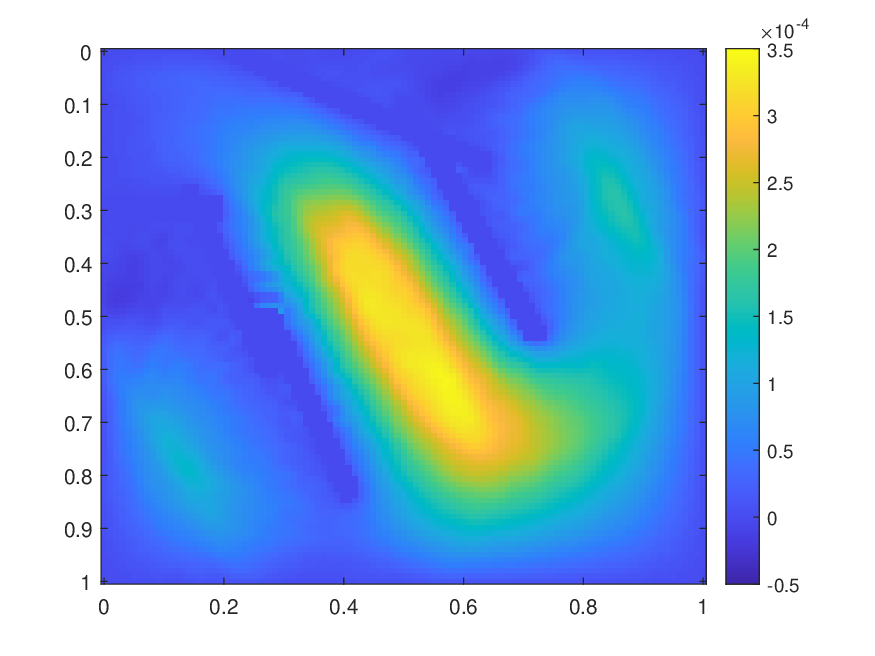}
        \caption{}
        \label{fig:image031}
    \end{subfigure}
    \begin{subfigure}[b]{0.37\textwidth}
        \includegraphics[width=\linewidth]{figure/f3rk.eps}
        \caption{}
        \label{fig:image041}
    \end{subfigure}
    \caption{Upper left: the first media $\kappa_1$. Upper right: reference solution for Case 1. Lower left: second order central difference solution for Case 1. Lower right: third order Runge-Kutta solution for Case 1.}
    \label{fig:k1}
\end{figure}
\begin{table}[H]
\centering
\begin{tabular}{c| c|c c c |c |c c c}
$\kappa$(contrast)&~& $e_2$&$e_a$&$e_b$&~& $e_2$&$e_a$&$e_b$ \\
\hline
$10^{7}$ &CD&3.92\% & 9.13\%&3.51\% &RK&3.55\%&8.54\%&3.46\%\\
\hline
$10^{6}$ &CD&3.92\% &9.13\%&3.51\%&RK&3.55\%&8.54\%&3.46\%\\
\hline
$10^{4}$ &CD&3.92\% &9.13\%&3.51\%&RK&3.55\%&8.54\%&3.46\%\\
\hline
\end{tabular}
\caption{The convergence of different contrast values $\kappa$ computed by the central difference scheme (CD) and the Runge-Kutta scheme (RK).}
 \label{tablek}
\end{table}
\begin{table}[H]
\centering
\begin{tabular}{c|c|c}
~ & Implicit & Partially explicit  \\
\hline
Central Difference& 2.90\% & 3.92\%  \\
\hline
Runge-Kutta & 2.78\% & 3.55\%  \\
\hline
\end{tabular}
\caption{The comparison of time discretization scheme.}
 \label{tableMKc300}
\end{table}
\begin{table}[H]
\centering
\begin{tabular}{c|c|c}
~ & Without mass lumping & With mass lumping  \\
\hline
Central Difference& 2.07\% & 3.92\%  \\
\hline
Runge-Kutta & 2.01\% & 3.55\%  \\
\hline
\end{tabular}
\caption{The comparison of the mass lumping scheme.}
 \label{tableMKc30}
\end{table}
 \begin{table}[H]
\centering
\begin{tabular}{c|c}
  ~ & Average convergence rate \\
\hline
Central Difference & 2.0947 \\
\hline
Runge-Kutta & 3.0656 \\
\hline
\end{tabular}
\caption{The comparison of the convergence rate.}
 \label{tableKPP20}
\end{table}
\begin{figure}[H]
    \centering
    \begin{subfigure}{0.42\textwidth}
        \includegraphics[width=\linewidth]{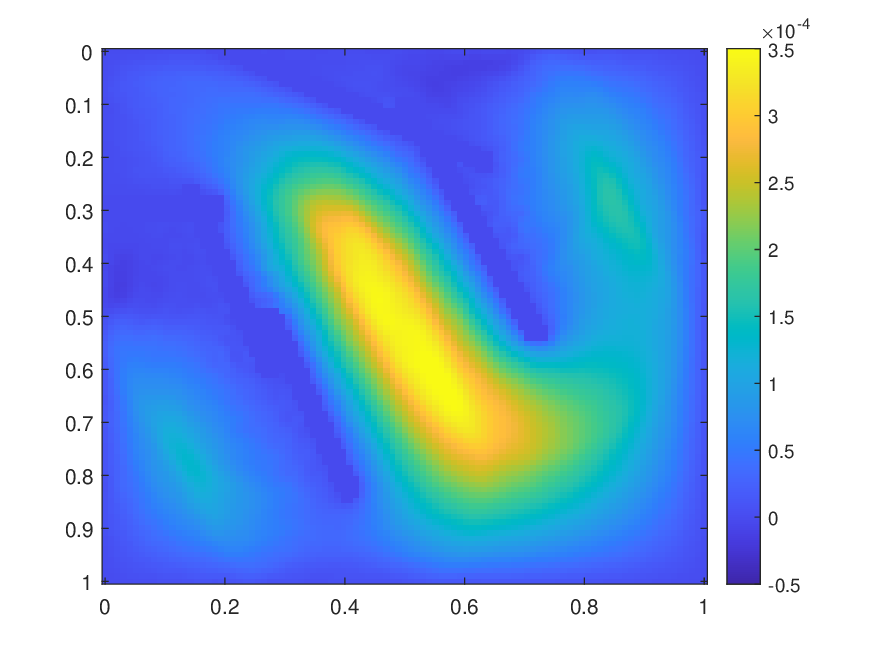}
        \caption{}
        \label{fig:image21c0}
    \end{subfigure}
    \begin{subfigure}{0.42\textwidth}
        \includegraphics[width=\linewidth]{figure/f3rk.eps}
        \caption{}
        \label{fig:image22c0}
    \end{subfigure}
    \caption{Left: solution via Runge-Kutta scheme without mass lumping. Right: solution via Runge-Kutta scheme with mass lumping.}
    \label{fig:mll}
\end{figure}
\begin{figure}[H]
     \centering
     \includegraphics[width=0.45\linewidth]{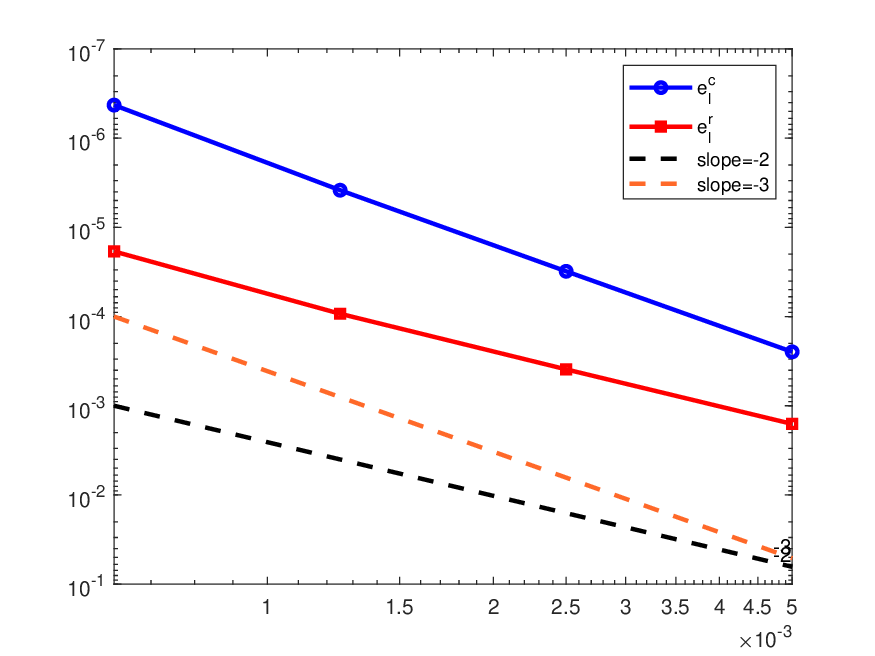}
     \caption{Convergence rate of the central difference and Runge-Kutta scheme.}
     \label{fig:cc21}
 \end{figure}

 \begin{table}[H]
\centering
\begin{tabular}{c|c|c}
~ & $e_2$ ($\tau =0.01/2^2$)  & $e_2$ ($\tau =0.01/2^3$) \\
\hline
Central Difference& 6.11\%  & 3.92\% \\
\hline
Runge-Kutta & 3.55\%& 3.55\% \\
\hline
\end{tabular}
\caption{$L^2$ error of different time step size $\tau$.}
 \label{tablec0}
\end{table}
\subsection{Case 2}
In our second numerical example, we take the source term $f_2$ and $\kappa_1$ to demonstrate the convergence of our proposed method with respect to the time step size $\tau$, coarse grid size $H$ and the number of oversampling layers $m$. 
The source term $f_2$ is chosen as 
\begin{equation}\label{f2}
\begin{aligned}
f_2(t,x,y)=10\left(-\frac{10tz_0-1}{z_0}\right)\mathrm{exp}\left(-\pi^2z_0^2\left(\frac{10tz_0-1}{z_0}\right)^2\right)\mathrm{exp}\left(-\frac{(x-0.5)^2+(y-0.5)^2}{160h^4}\right),
\end{aligned}
\end{equation}
for all $(t,x,y)\in \Omega$, where the fine grid parameter and the central frequency are chosen as $h=1/240$ and $z_0=2$ (Figure \ref{fig:C3}). Using the fully discrete scheme, we solve for the numerical solution at the final time $T=0.4$ with the time step size $\tau=2.5\times10^{-3}$ (see Figure \ref{fig:k}). Let $\kappa_{\text{cutoff}}=1$ in this case.  We plot the reference solution computed on the fine grid with the fully implicit scheme (upper right), the solution computed by CEM-GMsFEM with the partially-explicit scheme via mass lumping using the central difference scheme (lower left) and the Runge-Kutta scheme (lower right). In Table \ref{tableH}, the coarse grid size $H$ varies from $H=1/6$ to $H=1/24$, and the number of oversampling layers $m$ varies accordingly from $4$ to $7$, following the relationship $m\approx4\log(H)/log(1/6)$. It can be observed that the methods both result in good accuracy and the desired convergence error in the $L^2$ norm, energy norm and $b$-norm. 

Table \ref{table2} and Figure \ref{fig:f0112} depict that the second order central difference scheme converges slower than third order Runge-Kutta scheme, with average convergence rates approaching 2 and 3, respectively. The convergence rates with respect to the time step size $\tau$ and coarse grid size $H$ match our theoretical results.
It is also worth noting that after using mass lumping, the results closely resemble those of the original solution (see Figure \ref{fig:m0l} and Table \ref{tableMK30}). Specifically, the errors remain consistent, confirming that mass lumping does not degrade the accuracy of the numerical solution while significantly improving computational efficiency.
\begin{figure}[H]
    \centering
    \begin{subfigure}[b]{0.37\textwidth}
        \includegraphics[width=\linewidth]{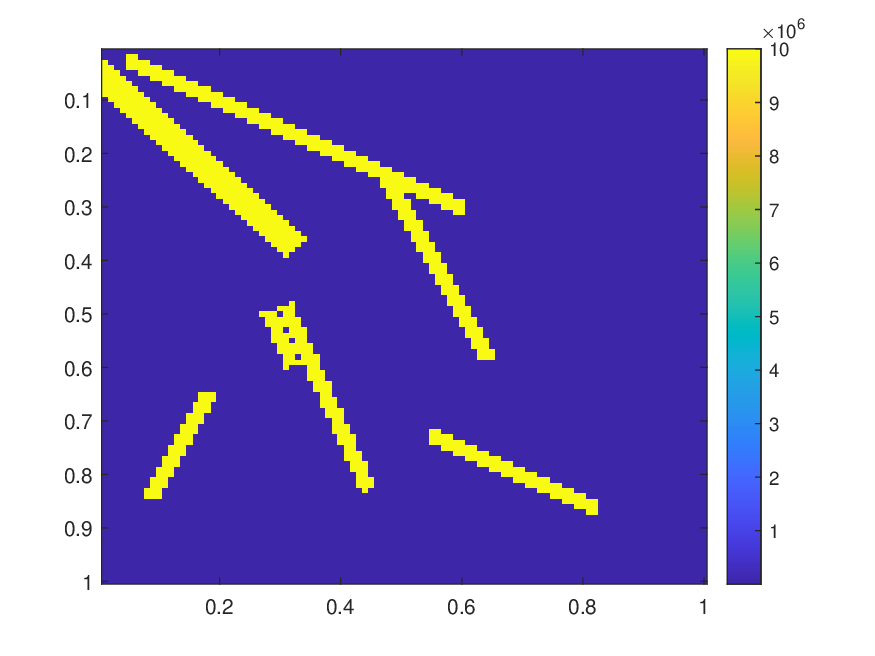}
        \caption{}
        \label{fig:image01}
    \end{subfigure}
    \begin{subfigure}[b]{0.37\textwidth}
        \includegraphics[width=\linewidth]{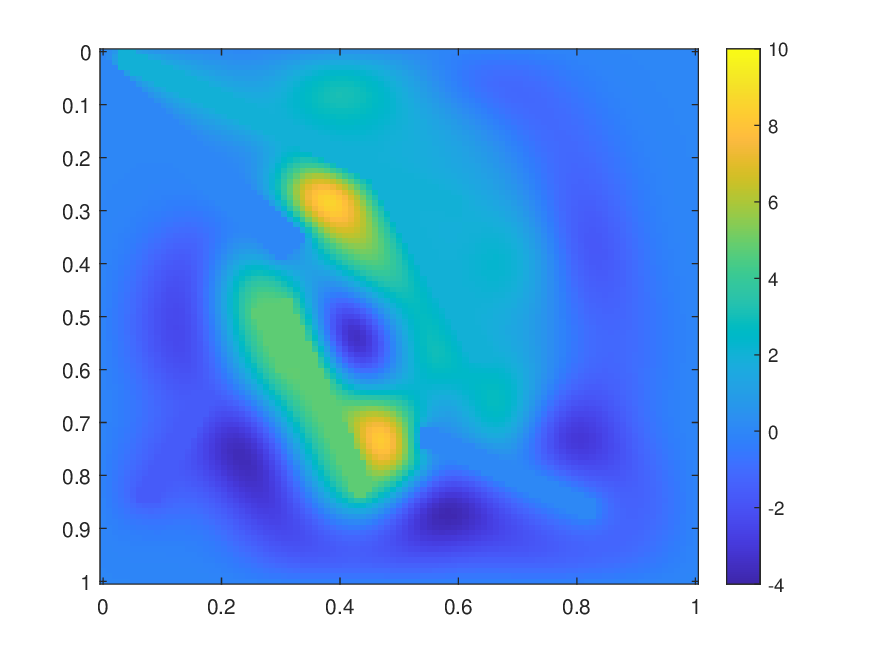}
        \caption{}
        \label{fig:image02}
    \end{subfigure}

    \begin{subfigure}[b]{0.37\textwidth}
        \includegraphics[width=\linewidth]{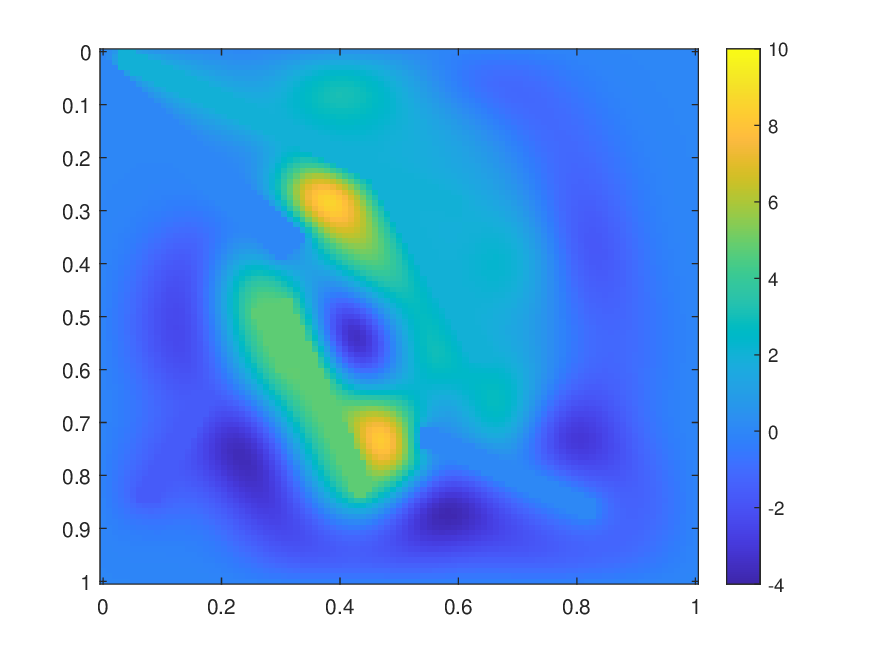}
        \caption{}
        \label{fig:image03}
    \end{subfigure}
    \begin{subfigure}[b]{0.37\textwidth}
        \includegraphics[width=\linewidth]{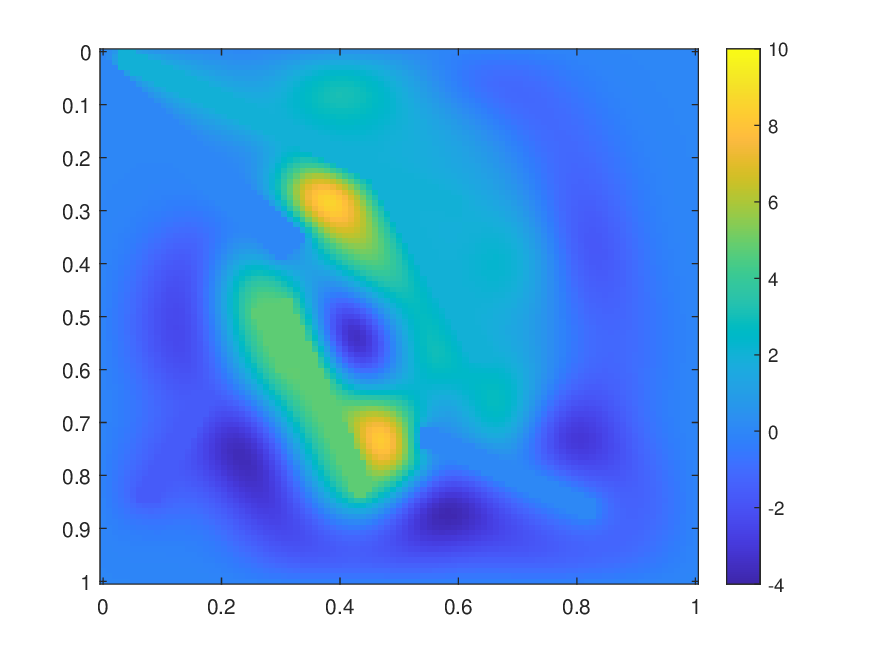}
        \caption{}
        \label{fig:image04}
    \end{subfigure}
    \caption{Upper left: the media $\kappa$ for Case 2. Upper right: reference solution for Case 2. Lower left: second order central difference solution for Case 2. Lower right: third order Runge-Kutta solution for Case 2.}
    \label{fig:k}
\end{figure}
\begin{table}[H]
\centering
\begin{tabular}{c c| c|c c c |c |c c c}
m &H&~& $e_2$&$e_a$&$e_b$&~& $e_2$&$e_a$&$e_b$ \\
\hline
4&$\frac{1}{6}$ &CD&71.14\% & 87.62\%&64.14\% &RK&64.70\%&87.62\%&64.14\%\\
\hline
6&$\frac{1}{12}$ &CD&22.96\% &43.98\%&21.62\%&RK&22.83\%&43.28\%&21.61\%\\
\hline
7& $\frac{1}{24}$ &CD& 3.58\%&8.56\% &3.35\%&RK&3.57\%&8.82\%&3.34\%\\
\hline
\end{tabular}
\caption{The convergence of different coarse grid sizes $H$ and the number of oversampling layers $m$ by the central difference scheme (CD) and Runge-Kutta scheme (RK).}
 \label{tableH}
\end{table} 
\begin{table}[H]
\centering
\begin{tabular}{c|c}
~ & Average Convergence Rate \\
\hline
Central Difference&  2.1087 \\
\hline
Runge-Kutta  & 3.0560 \\
\hline
\end{tabular}
\caption{The comparison of the convergence rate.}
 \label{table2}
\end{table} 
\begin{table}[H]
\centering
\begin{tabular}{c|c|c}
~ & Without mass lumping & With mass lumping  \\
\hline
Central Difference& 8.52\% & 11.20\%  \\
\hline
Runge-Kutta & 8.52\% & 11.02\%  \\
\hline
\end{tabular}
\caption{The comparison of mass lumping scheme.}
 \label{tableMK30}
\end{table}
\begin{figure}[H]
     \centering
     \includegraphics[width=0.5\linewidth]{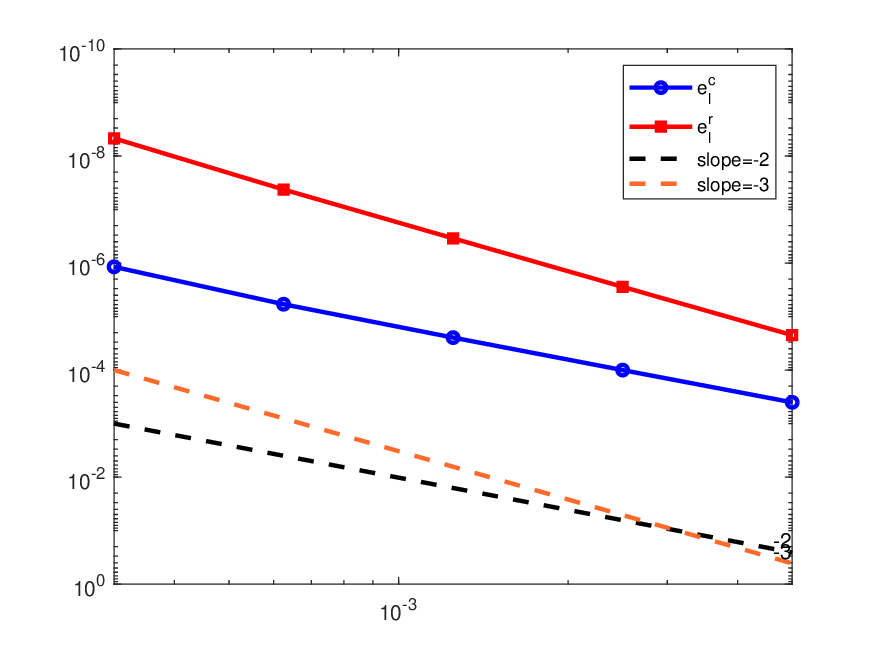}
     \caption{Convergence rate of the central difference and Runge-Kutta scheme.}
     \label{fig:f0112}
 \end{figure}
 \begin{figure}[htbp]
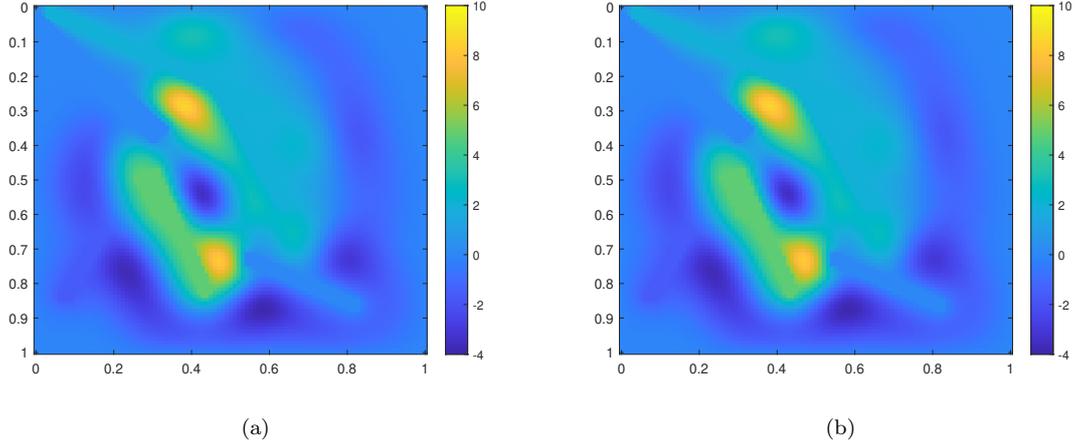

    \centering
    \begin{subfigure}{0.45\textwidth}
        \includegraphics[width=\linewidth]{figure/RK4.eps}
        \caption{}
        \label{fig:image210}
    \end{subfigure}
    \begin{subfigure}{0.45\textwidth}
        \includegraphics[width=\linewidth]{figure/massl.eps}
        \caption{}
        \label{fig:image220}
    \end{subfigure}
    \caption{Left: solution via Runge-Kutta scheme without mass lumping. Right: solution via Runge-Kutta scheme with mass lumping.}
    \label{fig:m0l}
\end{figure}

\subsection{Case 3}
In this numerical example, we consider a heterogeneous medium, specifically a modified Marmousi model (shown in subfigure (\ref{fig:image133}) in Figure \ref{fig:k2}). The source term is still $f_2$  (Figure \ref{fig:C3}). The fine grid parameter is chosen as $h=1/240$. We tested the convergence of different coarse grid sizes $H$ and frequency values $z_0$ of the source term shown in Table \ref{tableff}.
Using the fully discrete scheme, we solve for the numerical solution at the final time $T=0.1$ with the time step size $\tau=6.25\times10^{-4}$.  In this case, we set $\kappa_{\text{cutoff}}=35$. In Table \ref{tablee2}, the coarse grid size $H$ varies from $H=1/6$ to $H=1/24$, and the number of oversampling layers $m$ varies accordingly from $4$ to $7$, following the relationship $m\approx4\log(H)/log(1/6)$. It can be observed that the methods both result in good accuracy and the desired convergence error in the $L^2$ norm, energy norm and $b$-norm.
Figure \ref{fig:k2} depicts the reference solution computed on the fine grid with the fully implicit scheme (upper right), the solution computed by CEM-GMsFEM with the partially explicit scheme via mass lumping using the second order central difference scheme (lower left) and the Runge-Kutta scheme (lower right).
It can also be noticed that the fully implicit and partially explicit schemes provide similar results (see Table \ref{tablee}). Table \ref{tableMK3} and Figure \ref{fig:ml} show that mass lumping also provides a reasonable accuracy. From Figure \ref{fig:f0111}, we see clearly that we obtain the expected second order convergence for the central difference method and third order for the Runge-Kutta scheme (see Table \ref{tableMK2}). Furthermore, Table \ref{tablecee0} shows that the central difference scheme fails to converge unless the time step size is smaller than $\tau=0.01/2^5$. In comparison, Runge-Kutta converges faster, thereby saving computational resources.
\begin{table}[H]
\centering
\begin{tabular}{c|c| c c c|c| c c c }
\diagbox[width=3em]{$z_0$}{$H$}&~& $1/6$&$1/12$&$1/24$&~& $1/6$&$1/12$&$1/24$\\
\hline
$1/2$ &CD&55.14\% & 9.73\%&2.08\%&RK&54.04\% & 9.54\%&2.08\% \\
\hline
$5$ &CD&69.98\% & 14.99\%&3.51\%&RK&67.52\% & 14.31\%&3.17\% \\
\hline
$10$&CD & 82.63\% &26.12\%&6.13\%&RK&79.80\% & 26.06\%&5.97\%\\
\hline
\end{tabular}
\caption{$e_b$
of different coarse grid sizes $H$ and frequency values $z_0$ of the source term computed by the central difference scheme (CD) and Runge-Kutta scheme (RK).}
 \label{tableff}
 \end{table}
\begin{figure}[H]
    \centering
    \begin{subfigure}[b]{0.38\textwidth}
        \includegraphics[width=\linewidth]{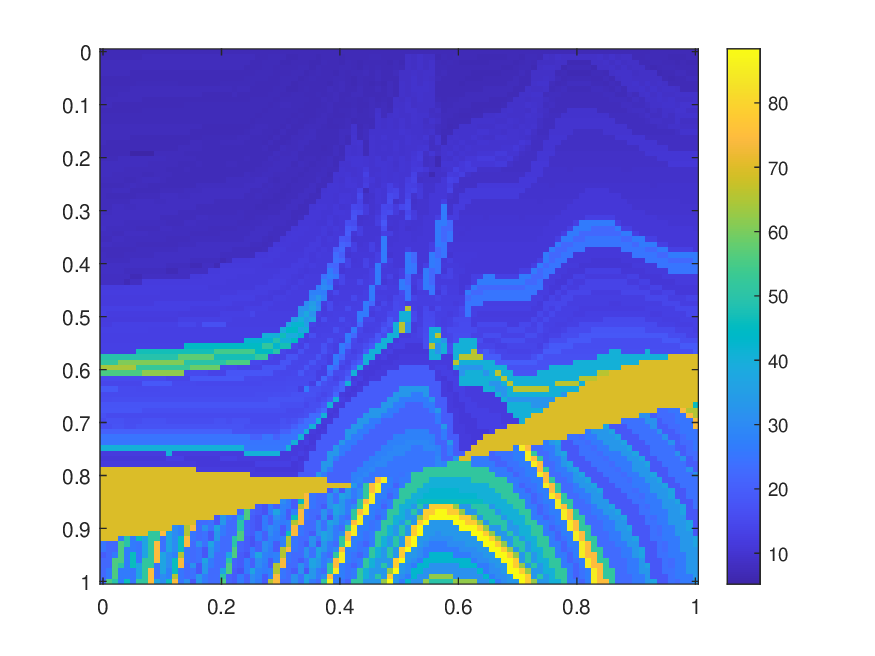}
        \caption{}
        \label{fig:image133}
    \end{subfigure}
    \begin{subfigure}[b]{0.38\textwidth}
        \includegraphics[width=\linewidth]{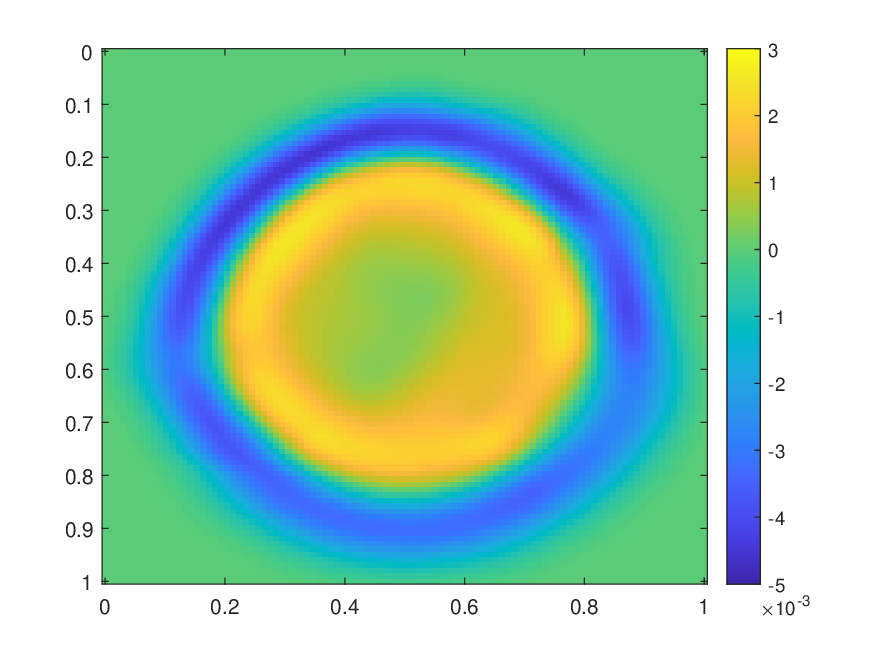}
        \caption{}
        \label{fig:image2333}
    \end{subfigure}
    \begin{subfigure}[b]{0.38\textwidth}
        \includegraphics[width=\linewidth]{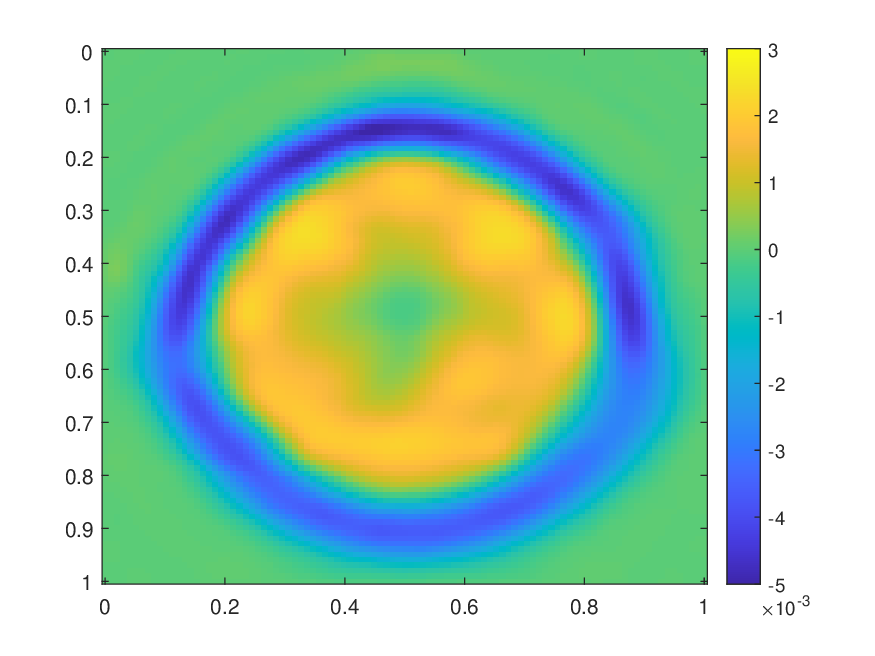}
        \caption{}
        \label{fig:image333}
    \end{subfigure}
     \begin{subfigure}[b]{0.38\textwidth}
        \includegraphics[width=\linewidth]{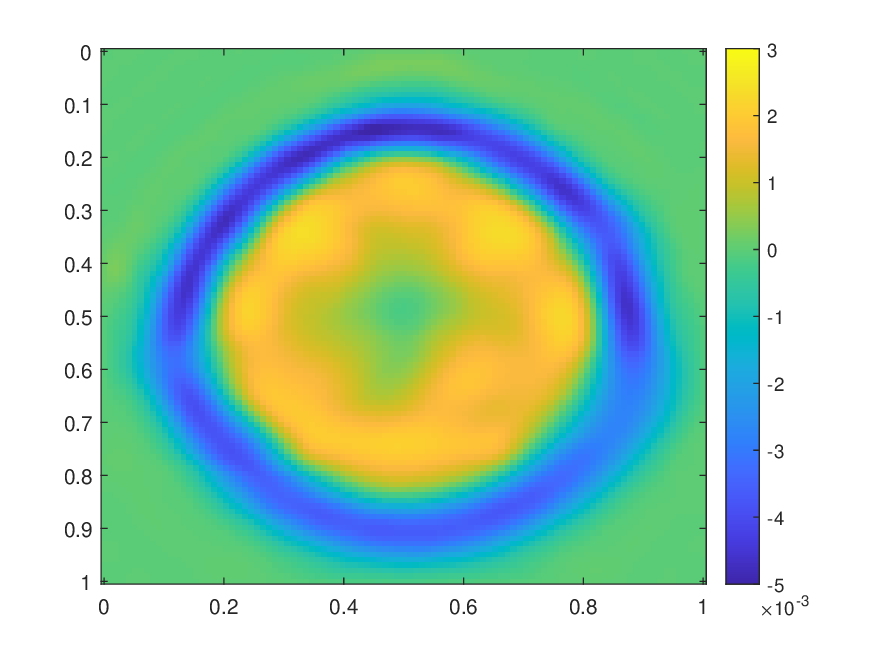}
        \caption{}
        \label{fig:image335}
    \end{subfigure}
    \caption{Upper left: modified Marmousi model. Upper right: reference solution for Case 3. Lower left: second order central difference solution for Case 3. Lower right: third order Runge-Kutta solution for Case 3.}
    \label{fig:k2}
\end{figure}
\begin{table}[H]
\centering
\begin{tabular}{c|c|c}
~ & Implicit & Partially explicit  \\
\hline
Central Difference& 11.44\% & 12.84\%  \\
\hline
Runge-Kutta & 10.85\% & 11.43\%  \\
\hline
\end{tabular}
\caption{The comparison of the time discretization scheme.}
 \label{tablee}
\end{table}
\begin{table}[H]
\centering
\begin{tabular}{c|c|c}
~ & Without mass lumping & With mass lumping  \\
\hline
Central Difference& 10.77\% & 12.84\%  \\
\hline
Runge-Kutta & 10.21\% & 11.43\%  \\
\hline
\end{tabular}
\caption{The comparison of the mass lumping scheme.}
 \label{tableMK3}
\end{table}
\begin{figure}[H]
    \centering
    \begin{subfigure}{0.45\textwidth}
        \includegraphics[width=\linewidth]{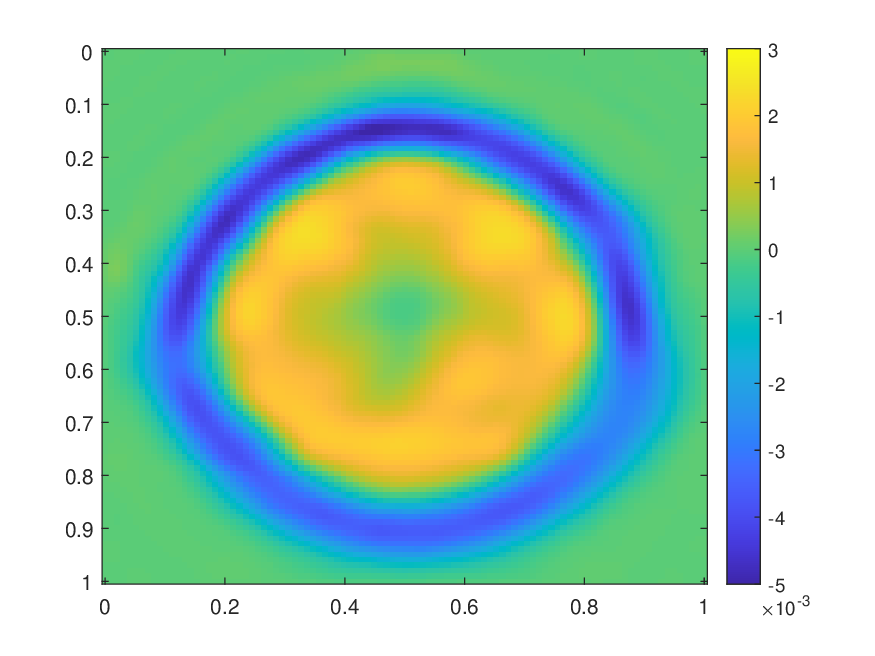}
        \caption{}
        \label{fig:image21}
    \end{subfigure}
    \begin{subfigure}{0.45\textwidth}
        \includegraphics[width=\linewidth]{figure/MkCD2.eps}
        \caption{}
        \label{fig:image22}
    \end{subfigure}
    \caption{Left: solution via Runge-Kutta scheme without mass lumping. Right: solution via Runge-Kutta scheme with mass lumping.}
    \label{fig:ml}
\end{figure}

\begin{figure}[H]
     \centering
     \includegraphics[width=0.5\linewidth]{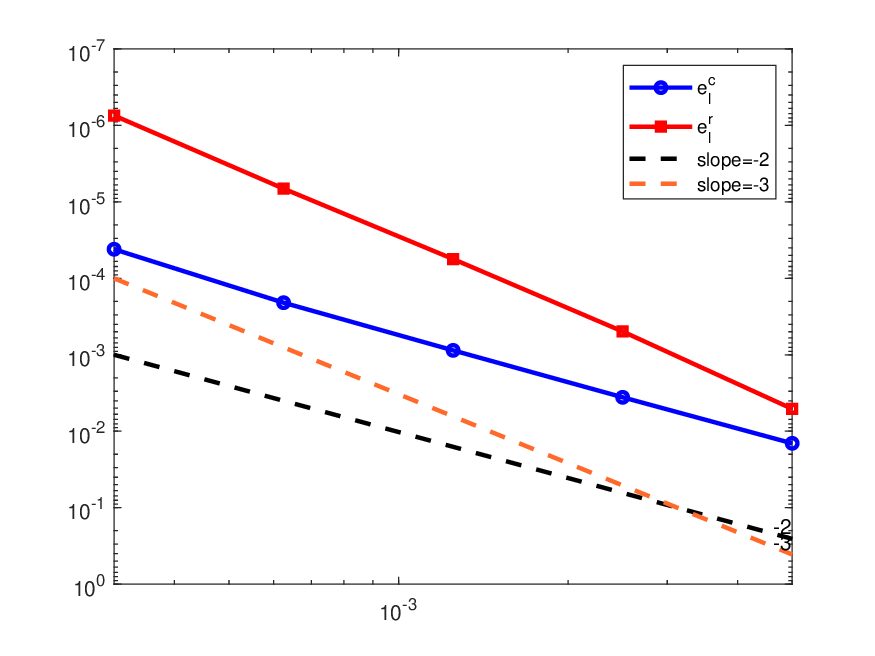}
     \caption{Convergence rate of central difference and Runge-Kutta scheme.}
     \label{fig:f0111}
 \end{figure}
 \begin{table}[H]
\centering
\begin{tabular}{c c| c|c c c |c |c c c}
m &H&~& $e_2$&$e_a$&$e_b$&~& $e_2$&$e_a$&$e_b$ \\
\hline
4&$1/6$ &CD&59.91\% & 94.12\%&41.66\% &RK&59.09\%&94.03\%&40.59\%\\
\hline
6&$1/12$ &CD&12.84\% &18.71\%&10.89\%&RK&11.43\%&18.45\%&10.30\%\\
\hline
7& $1/24$ &CD& 1.73\%&4.36\% &1.99\%&RK&1.72\%&4.31\%&1.95\%\\
\hline
\end{tabular}
\caption{The convergence of different coarse grid sizes $H$ and the number of oversampling layers $m$ computed by the central difference scheme (CD) and Runge-Kutta scheme (RK).}
 \label{tablee2}
\end{table}
\begin{table}[H]
\centering
\begin{tabular}{c|c}
~ &  Average convergence rate \\
\hline
Central Difference& 2.1039 \\
\hline
Runge-Kutta &  3.0825 \\
\hline
\end{tabular}
\caption{The comparison of the convergence rate.}
 \label{tableMK2}
\end{table}
\begin{table}[H]
\centering
\begin{tabular}{c|c|c|c}
~ & $e_2$ ($\tau =0.01/2^3$)  & $e_2$ ($\tau =0.01/2^4$) & $e_2$ ($\tau =0.01/2^5$) \\
\hline
Central Difference& 17.91\%  & 12.84\% & 11.82\% \\
\hline
Runge-Kutta & 11.55\%& 11.43\% & 11.43\%\\
\hline
\end{tabular}
\caption{$L^2$ error of different time step size $\tau$.}
 \label{tablecee0}
\end{table}

\section{Conclusion}\label{007}
In this paper, we designed a contrast-independent, partially explicit time discretization method for wave equations using mass lumping. The proposed approach employs temporal splitting based on a multiscale decomposition of the approximation space. Initially, we introduced two spatial subspaces corresponding to fast and slow time scales. We present temporal splitting algorithms using the aforementioned spatial multiscale methods for our system. Then, we derive the stability and convergence conditions of the proposed method and demonstrate that these conditions are contrast-independent.
To address the issue of system coupling, we introduce a mass lumping scheme via the diagonalization technique, which avoids matrix inversion procedures and significantly improves computational efficiency, especially in the explicit part. Furthermore, after decoupling the resulting system, higher-order time discretization techniques, such as the Runge-Kutta method, can be applied to achieve better accuracy within the same number of time steps.
We present several numerical examples that show the proposed method yields outcomes very similar to those of fully implicit schemes. Our numerical results indicate that only a minimal number of auxiliary basis functions is required to achieve good accuracy, independent of contrast. Additionally, the convergence rate is influenced by the time step size, the coarse grid size, and the time-splitting scheme used, all of which strongly confirm our theoretical findings.

\section*{Acknowledgments}
W.T. Leung is partially supported by the Hong Kong RGC Early Career Scheme 21307223.
\appendix
\section{Proof of Lemma \ref{lemma1}} \label{appl1}
\begin{proof}
    We have 
\begin{equation}\label{1.111}
    b\left(u_{H}^{n+1}-2u_{H}^{n}+u_{H}^{n-1},w\right)+\frac{\tau^2}{2}a\left(u_{H,1}^{n+1}+u_{H,1}^{n-1}+2u_{H,2}^{n},w\right)=0~ \forall w \in V_{H,1},
\end{equation}
\begin{equation}\label{1.112}
    b\left(u_{H}^{n+1}-2u_{H}^{n}+u_{H}^{n-1},w\right)+\tau^2 a\left(u_{H,1}^{n}+u_{H,2}^{n},w\right)=0~\forall w \in V_{H,2}.
\end{equation}
We consider $w=u_{H,1}^{n+1}-u_{H,1}^{n-1}$ in the first equation and $w=u_{H,2}^{n+1}-u_{H,2}^{n-1}$ in the second equation and obtain the following equations
\begin{equation}\label{1.113}
\begin{aligned}
    &b\left(u_{H}^{n+1}-2u_{H}^{n}+u_{H}^{n-1},u_{H,1}^{n+1}-u_{H,1}^{n-1}\right)\\
    +&\frac{\tau^2}{2}a\left(u_{H,1}^{n+1}+u_{H,1}^{n-1}+2u_{H,2}^{n},u_{H,1}^{n+1}-u_{H,1}^{n-1}\right)=0~ \forall w \in V_{H,1},\\
    &b\left(u_{H}^{n+1}-2u_{H}^{n}+u_{H}^{n-1},u_{H,2}^{n+1}-u_{H,2}^{n-1}\right)\\
    +&\tau^2 a\left(u_{H,1}^{n}+u_{H,2}^{n},u_{H,2}^{n+1}-u_{H,2}^{n-1}\right)=0~\forall w \in V_{H,2}.
\end{aligned}
\end{equation}
The sum of the first terms on the left hand sides can be estimated in the following way 
\begin{equation}\label{1.115}
\begin{aligned}
    &\sum_{i=1,2}b(u_H^{n+1}-2u_H^n+u_H^{n-1},u_{H,i}^{n+1}-u_{H,i}^{n-1})\\
    =&b(u_H^{n+1}-2u_H^n+u_H^{n-1},u_{H}^{n+1}-u_{H}^{n-1})\\
    =&\|u_H^{n+1}-u_H^n\|_b^2-\|u_H^{n}-u_H^{n-1}\|_b^2.
\end{aligned}
\end{equation}
Next, we estimate the terms involving the bilinear form $a$, it has
\begin{equation}\label{1.116}
\begin{aligned}
     &\frac{1}{2} a\big(u_{H,1}^{n+1}+u_{H,1}^{n-1}+2u_{H,2}^{n},u_{H,1}^{n+1}-u_{H,1}^{n-1}\big)\\
     =& \frac{1}{2}a\big(u_{H,1}^{n+1}+u_{H,1}^{n-1},u_{H,1}^{n+1}-u_{H,1}^{n-1}\big)+a\big(u_{H,2}^{n},u_{H,1}^{n+1}-u_{H,1}^{n-1}\big)
\end{aligned}
\end{equation}
and
\begin{equation}\label{1.117}
\begin{aligned}
      &a\big (u_{H,1}^{n}+u_{H,2}^{n},u_{H,2}^{n+1}-u_{H,2}^{n-1}\big)\\
      =&a\big(u_{H,1}^{n},u_{H,2}^{n+1}-u_{H,2}^{n-1}\big)+a\big(u_{H,2}^{n},u_{H,2}^{n+1}-u_{H,2}^{n-1}\big).
\end{aligned}
\end{equation}
Thus, we have 
\begin{equation}\label{1.118}
\begin{aligned}
      &\frac{\tau^2}{2}a\big(u_{H,1}^{n+1}+u_{H,1}^{n-1}+2u_{H,2}^{n},u_{H,1}^{n+1}-u_{H,1}^{n-1}\big)+\tau^2 a\big(u_{H,1}^{n}+u_{H,2}^{n},u_{H,2}^{n+1}-u_{H,2}^{n-1}\big)\\
      =&\frac{\tau^2}{2}Q_1+\tau^2Q_2+\tau^2Q_3,
\end{aligned}
\end{equation}
where
\begin{equation}\label{1.119}
\begin{aligned}
      Q_1&=a\big(u_{H,1}^{n+1}+u_{H,1}^{n-1},u_{H,1}^{n+1}-u_{H,1}^{n-1}\big),\\
      Q_2&=a\big(u_{H,2}^{n},u_{H,1}^{n+1}-u_{H,1}^{n-1})+a\big(u_{H,1}^{n},u_{H,2}^{n+1}-u_{H,2}^{n-1}\big),\\
      Q_3&=a\big(u_{H,2}^{n},u_{H,2}^{n+1}-u_{H,2}^{n-1}\big).
\end{aligned}
\end{equation}
To estimate $Q_1$, we have
\begin{equation}\label{1.120}
\begin{aligned}
      Q_1=&a\big(u_{H,1}^{n+1}+u_{H,1}^{n-1},u_{H,1}^{n+1}-u_{H,1}^{n-1}\big)=\big\|u_{H,1}^{n+1}\big\|_a^2-\big\|u_{H,1}^{n-1}\big\|_a^2\\
      =&\Big(\big\|u_{H,1}^{n+1}\big\|_a^2+\big\|u_{H,1}^{n}\big\|_a^2\Big)-\Big(\big\|u_{H,1}^{n}\big\|_a^2+\big\|u_{H,1}^{n-1}\big\|_a^2\Big).
\end{aligned}
\end{equation}
We next estimate $Q_2$ and have
\begin{equation}\label{1.121}
\begin{aligned}
      Q_2=&a\big(u_{H,2}^{n},u_{H,1}^{n+1}-u_{H,1}^{n-1}\big)+a\big(u_{H,1}^{n},u_{H,2}^{n+1}-u_{H,2}^{n-1}\big)\\
      =&a\big(u_{H,2}^{n},u_{H,1}^{n+1}\big)+a\big(u_{H,1}^{n},u_{H,2}^{n+1}\big)-a\big(u_{H,2}^{n},u_{H,1}^{n-1}\big)-a\big(u_{H,1}^{n},u_{H,2}^{n-1}\big).
\end{aligned}
\end{equation}
Since $a(\cdot,\cdot)$ is symmetric, we have
\begin{equation}\label{1.122}
\begin{aligned}      Q_2=\big(a(u_{H,2}^{n+1},u_{H,1}^{n})+a(u_{H,1}^{n+1},u_{H,2}^{n})\big)-\big(a(u_{H,2}^{n},u_{H,1}^{n-1})+a(u_{H,1}^{n},u_{H,2}^{n-1})\big).
\end{aligned}
\end{equation}
To estimate $Q_3$, we have 
\begin{equation}\label{1.123}
\begin{aligned}
      Q_3=a\big(u_{H,2}^{n},u_{H,2}^{n+1}-u_{H,2}^{n-1}\big)=a\big(u_{H,2}^{n+1},u_{H,2}^{n}\big)-a\big(u_{H,2}^{n},u_{H,2}^{n-1}\big).
\end{aligned}
\end{equation}
We also have 
\begin{equation}\label{1.124}
\begin{aligned}
      a\big(u_{H,2}^{n+1},u_{H,2}^{n}\big)=\frac{1}{2}\left(\|u_{H,2}^{n+1}\|_a^2+\|u_{H,2}^{n}\|_a^2-\|u_{H,2}^{n+1}-u_{H,2}^{n}\|_a^2\right)
\end{aligned}
\end{equation}
and
\begin{equation}\label{1.125}
\begin{aligned}
      a(u_{H,2}^{n},u_{H,2}^{n-1})=\frac{1}{2}\left(\|u_{H,2}^{n}\|_a^2+\|u_{H,2}^{n-1}\|_a^2-\|u_{H,2}^{n}-u_{H,2}^{n-1}\|_a^2\right).
\end{aligned}
\end{equation}
Thus, we have 
\begin{equation}\label{1.126}
\begin{aligned}
&\frac{\tau^2}{2} a\left(u_{H,1}^{n+1}+u_{H,1}^{n-1}+2u_{H,2}^{n},u_{H,1}^{n+1}-u_{H,1}^{n-1}\right) \\
&+\tau^2a\left(u_{H,1}^{n}+u_{H,2}^{n},u_{H,2}^{n+1}-u_{H,2}^{n-1}\right)\\
=&\frac{\tau^2}{2}\left(\|u_{H,1}^{n+1}\|_a^2+\|u_{H,1}^{n}\|_a^2\right)-\frac{\tau^2}{2}\left(\|u_{H,1}^{n}\|_a^2+\|u_{H,1}^{n-1}\|_a^2\right)\\
&+\tau^2\Big(a\big(u_{H,1}^{n+1},u_{H,2}^{n}\big)+a\big(u_{H,2}^{n+1},u_{H,1}^{n}\big)\Big)\\
&+\frac{\tau^2}{2}\left(\|u_{H,2}^{n+1}\|_a^2+\|u_{H,2}^{n}\|_a^2-\|u_{H,2}^{n+1}-u_{H,2}^{n}\|_a^2\right)\\
&-\tau^2\left(a(u_{H,1}^{n},u_{H,2}^{n-1})+a(u_{H,2}^{n},u_{H,1}^{n-1})\right)\\
&-\frac{\tau^2}{2}\left(\|u_{H,2}^{n}\|_a^2+\|u_{H,2}^{n-1}\|_a^2-\|u_{H,2}^{n}-u_{H,2}^{n-1}\|_a^2\right).
\end{aligned}
\end{equation}
By definition, the conclusion can be derived.
\end{proof}
\section{Proof of Theorem \ref{2.01}.}\label{appb}
\begin{proof}
    By the definition of $R^n_1,R^n_2$ and $R^n_3$, we have
\begin{equation}\label{2.3}
\begin{aligned}R_1^n=&\sum_{k=1}^n\left\|\pi\left(\frac{\eta^{n+1}-2\eta^n+\eta^{n-1}}{\tau^2}\right)\right\|_{L^2(\Omega)}\\
     &+\sum_{k=1}^n\left\|\pi\left(\frac{\partial^2u}{\partial t^2}(\cdot,t_n)-\frac{u^{n+1}-2u^n+u^{n-1}}{\tau^2}\right)\right\|_{L^2(\Omega)},\\
     R_2^n=&\tau\sum_{k=1}^{n-1}\left\|(I-\pi)(\frac{f^{k+1}-f^k}{\tau})\right\|_{L^2(\Omega)}+\|(I-\pi)(f^1)\|_{L^2(\Omega)}\\
     &+\|(I-\pi)(f^n)\|_{L^2(\Omega)},\\
     R_3^n=&\sum_{k=1}^n\left\|\left(\frac{P_{H,1}u^{n+1}-2P_{H,1}u^n+P_{H,1}u^{n-1}}{2}\right)\right\|_a.
\end{aligned}
\end{equation}
By using Taylor's expansion and (\ref{L.2}), it can be estimated
\begin{equation}\label{2.32}
\begin{aligned}
     \tau (R_1^n+R_3^n)+\Lambda^{-\frac{1}{2}}HR_2^n\leq C(\Lambda^{-1}H^2+\tau^2).
     \end{aligned}
\end{equation}
In particular, choose $g=f-\frac{\partial^2u}{\partial t^2}$ and $u^n=G(g(\cdot,t_n)).$
By Taylor's theorem, it has
\begin{equation}\label{T1}
\begin{aligned}
     g(\cdot,t_n+t)=g(\cdot,t_n)+t\frac{\partial g}{\partial t}(\cdot,t_n)+\int_{t_n}^{t_n+t}s\frac{\partial^2g}{\partial t^2}(\cdot,s)ds.
\end{aligned}
\end{equation}
Integrating from $t=-\tau$ to $t=\tau$, we have
\begin{equation}\label{T1.1}
\begin{aligned}
     \|g(\cdot,t_n)\|_{L^2(\Omega)}\leq \frac{1}{2\tau}\|g\|_{L^1(t_{n-1},t_{n+1};L^2(\Omega))}+\frac{\tau}{2}\|\frac{\partial g}{\partial t}\|_{L^1(t_{n-1},t_{n+1};L^2(\Omega))}.
\end{aligned}
\end{equation}
From (\ref{L.2}), it can be derived that
\begin{equation}\label{T1.2}
\begin{aligned}
     \|\eta\|_{L^2(\Omega)}\leq C\Lambda^{-1}H^2\left(\|g\|_{C([0,T];L^2(\Omega))}+\tau^2\left\|\frac{\partial^2g}{\partial t^2}\right\|_{C([0,T];L^2(\Omega))}\right).
\end{aligned}
\end{equation}
Then, since
\begin{equation}\label{T2}
\begin{aligned}
     g(\cdot,t_n+t)=&g(\cdot,t_n)+t\frac{\partial g}{\partial t}(\cdot,t_n)+\frac{t^2}{2}\frac{\partial^2g}{\partial t^2}(\cdot,t_n)\\
     &+\frac{t^3}{6}\frac{\partial^3g}{\partial t^3}(\cdot,t_n)+\int_{t_n}^{t_n+t}\frac{s^3}{6}\frac{\partial^4g}{\partial t^4}(\cdot,s)ds.
\end{aligned}
\end{equation}
Taking $t=\pm\tau$, we have 
\begin{equation}\label{T2.1}
\begin{aligned}
     &\|g(\cdot,t_{n+1})-2g(\cdot,t_{n})+g(\cdot,t_{n-1})\|_{L^2(\Omega)}\\
     &\leq \tau^2\|\frac{\partial^2g}{\partial t^2}(\cdot,t_n)\|_{L^2(\Omega)}+\frac{s^3}{3}\|\frac{\partial^4g}{\partial t^4}\|_{L^1(t_{n-1},t_{n+1};L^2(\Omega))}.
\end{aligned}
\end{equation}
Replacing $g$ by $\frac{\partial^2g}{\partial t^2}$ in \ref{T1.1}, and using (\ref{L.2}), it has
\begin{equation}\label{T2.2}
\begin{aligned}
     &\left\|\frac{\eta^{n+1}-2\eta^n+\eta^{n-1}}{\tau^2}\right\|_{L^2(\Omega)}\\
     &\leq C\Lambda^{-1}H^2\left(\|\frac{\partial^2g}{\partial t^2}(\cdot,t_n)\|_{C([0,T];L^2(\Omega))}+\frac{s^3}{3}\|\frac{\partial^4g}{\partial t^4}\|_{C([0,T];L^2(\Omega))}\right).
\end{aligned}
\end{equation}
Thus, 
\begin{equation}\label{T2.3}
\begin{aligned}
     &\sum_{k=1}^n\left\|\pi\left(\frac{\eta^{n+1}-2\eta^n+\eta^{n-1}}{\tau^2}\right)\right\|_{L^2(\Omega)}\\
     &\leq C\tau^{-1}\Lambda^{-1}H^2\left\|\frac{\partial^2 f}{\partial t^2}-\frac{\partial^4 u}{\partial t^4}\right\|_{C([0,T];L^2(\Omega))}+C\tau\left\|\frac{\partial^4 u}{\partial t^4}\right\|_{C([0,T];L^2(\Omega))}.
\end{aligned}
\end{equation}
Similarly, we estimate the second term in $R_1^n$ using Taylor's expansion
\begin{equation}\label{T2.4}
\begin{aligned}
     \sum_{k=1}^n\left\|\pi\left(\frac{\partial^2u}{\partial t^2}(\cdot,t_n)-\frac{u^{n+1}-2u^n+u^{n-1}}{\tau^2}\right)\right\|_{L^2(\Omega)}\leq C\tau\left\|\frac{\partial^4 u}{\partial t^4}\right\|_{C([0,T];L^2(\Omega))}.
\end{aligned}
\end{equation}
Using $\gamma_a = \sup_{u_2\in V_{H,2},u_1\in V_{H,1}}\cfrac{a(u_1,u_2)}{\|u_1\|_a\|u_2\|_a}<1$, we have
\begin{equation}
\begin{aligned}
    s_1^n&=-P_{H,1}\left(\frac{u^{n+1}-2u^n+u^{n-1}}{2}\right),\\
    \|s_1^n\|_a&\leq \frac{1}{2(1-\gamma_a)}\|u^{n+1}-2u^n+u^{n-1}\|_a\leq\frac{C\tau^2}{2(1-\gamma_a)}\sup_{t\in[0,T]}\|\frac{\partial^2u}{\partial t^2}(\cdot,t)\|_a
\end{aligned}
\end{equation}
\end{proof}

\section{Proof of Theorem \ref{2.02}}\label{appd}
\begin{proof}
Recalling the definition of $u$ and $u_H$ and noting that $a(\eta^n,w)=0,$ for any $w\in V_H^{(m)}$, we have
\begin{equation}\label{2.210}
\begin{aligned}
     &b\left(\frac{\delta^{n+1}-2\delta^n+\delta^{n-1}}{\tau^2},w\right)+a\left(\frac{\delta_1^{n+1}+\delta_1^{n-1}}{2}+\delta_2^n,w\right)\\
     &=(r_1^n,w)+a(s_1^n,w)~\forall w\in V_{H,1},\\
     &b\left(\frac{\delta^{n+1}-2\delta^n+\delta^{n-1}}{\tau^2},w\right)+a\left(\delta_1^{n}+\delta_2^{n},w\right)=(r_2^n,w)~\forall w\in V_{H,2},
\end{aligned}
\end{equation}
where 
\begin{equation}\label{2.220}
\begin{aligned}
     r^n=r_1^n+r_2^n=&\pi\left(\frac{\partial^2u}{\partial t^2}(\cdot,t_n)-\frac{u^{n+1}-2u^n+u^{n-1}}{\tau^2}\right) + (I-\pi)\left(\cfrac{\partial^2u^n}{\partial t^2}\right)\\
     &+\pi\left(\frac{\eta^{n+1}-2\eta^n+\eta^{n-1}}{\tau^2}\right)+(I-\pi)(f^n),\\
     s_1^n=&-\left(\frac{P_{H,1}u^{n+1}-2P_{H,1}u^n+P_{H,1}u^{n-1}}{2}\right).
\end{aligned}
\end{equation}
    Denote $\Delta^n=\tau\sum_{k=1}^n\delta^k.$ Using the telescoping sum over $(\ref{2.210})$, we have 
\begin{equation}\label{2.5}
\begin{aligned}
     b\left(\frac{\delta^{n+1}-\delta^{n-1}}{\tau},w\right)-b\left(\frac{\delta^{1}-\delta^{0}}{\tau},w\right)+a\left(\frac{\Delta_1^{n+1}+\Delta_1^{n-1}}{2}+\Delta_2^{n},w\right)\\
     =\tau\sum_{k=1}^n(r^k,w)+\tau\sum_{k=1}^n a(s_1^k,w)~\forall w\in V_{H,1},
\end{aligned}
\end{equation}
\begin{equation}\label{2.6}
\begin{aligned}
     b\left(\frac{\delta^{n+1}-\delta^{n-1}}{\tau},w\right)-b\left(\frac{\delta^{1}-\delta^{0}}{\tau},w\right)+a\left(\Delta_1^{n}+\Delta_2^{n},w\right)\\
     =\tau\sum_{k=1}^n(r^k,w)~\forall w\in V_{H,2}.\\
\end{aligned}
\end{equation}
Taking $w=\Delta_1^{n+1}-\Delta_1^{n-1}=\tau(\delta_1^{n+1}+\delta_1^n)\in V_{H,1}^{(m)}$ in $(\ref{2.5})$, and $w=\Delta_2^{n+1}-\Delta_2^{n-1}=\tau(\delta_2^{n+1}+\delta_2^n)\in V_{H,2}^{(m)}$ in $(\ref{2.6})$, it can be inferred that
\begin{equation}\label{2.7}
\begin{aligned}
     &\|\delta_1^{n+1}\|_b^2-\|\delta_1^{n}\|_b^2+a\left(\Delta_2^n,\Delta_1^{n+1}-\Delta_1^n\right)\\
     &+a\left(\frac{\Delta_1^{n+1}+\Delta_1^{n-1}}{2},\Delta_1^{n+1}-\Delta_1^n\right)\\
     =&b(\delta_1^1-\delta_1^0,\delta_1^{n+1}+\delta_1^{n})\\
     &+\sum_{k=1}^n(r^k,\tau^2(\delta_1^{n+1}+\delta_1^{n}))+\sum_{k=1}^na(s_1^k,\tau^2(\delta_1^{n+1}+\delta_1^{n})),
\end{aligned}
\end{equation}
and
\begin{equation}\label{2.8}
\begin{aligned}
     &\|\delta_2^{n+1}\|_b^2-\|\delta_2^{n}\|_b^2+a(\Delta_1^n,\Delta_2^{n+1}-\Delta_2^n)+a\left(\Delta_2^n,\Delta_2^{n+1}-\Delta_2^n\right)\\
     =&\sum_{k=1}^n\Big(r^k,\tau^2\big(\delta_2^{n+1}+\delta_2^{n}\big)\Big)+b\Big(\delta_2^1-\delta_2^0,\delta_2^{n+1}+\delta_2^{n}\Big).
\end{aligned}
\end{equation}
Rewrite $(\ref{2.7})$ as follows
\begin{equation}\label{2.9}
\begin{aligned}
     &\|\delta_1^{n+1}\|_b^2-\|\delta_1^{n}\|_b^2++a\left(\Delta_2^n,\Delta_1^{n+1}-\Delta_1^{n-1}\right)\\
     &+a\left(\frac{\Delta_1^{n+1}-2\Delta_1^{n}+\Delta_1^{n-1}}{2}+\Delta_1^{n},\Delta_1^{n+1}-\Delta_1^{n-1}\right)\\
     =&\sum_{k=1}^n\Big(r^k,\tau^2\big(\delta_1^{n+1}+\delta_1^{n}\big)\Big)+\sum_{k=1}^na\Big(s_1^k,\tau^2\big(\delta_1^{n+1}+\delta_1^{n}\big)\Big)\\
     &+b\Big(\delta_1^1-\delta_1^0,\delta_1^{n+1}+\delta_1^{n}\Big).
\end{aligned}
\end{equation}
Summing $(\ref{2.8})$ and $(\ref{2.9})$, it can be derived
\begin{equation}\label{2.10}
\begin{aligned}
     &\|\delta_1^{n+1}\|_b^2-\|\delta_1^{n}\|_b^2+\|\delta_2^{n+1}\|_b^2-\|\delta_2^{n}\|_b^2\\
     &+a\left(\frac{\Delta_1^{n+1}-2\Delta_1^{n}+\Delta_1^{n-1}}{2},\Delta_1^{n+1}-\Delta_1^{n-1}\right)
     +a(\Delta^{n},\Delta^{n+1}-\Delta^{n-1})\\
     =&b\big(\delta_1^1-\delta_1^0,\delta_1^{n+1}+\delta_1^{n}\big)+\sum_{k=1}^n\big(r^k,\tau^2(\delta_1^{n+1}+\delta_1^{n})\big)+\sum_{k=1}^na\big(s_1^k,\tau^2(\delta_1^{n+1}+\delta_1^{n})\big)\\
     &+b\big(\delta_2^1-\delta_2^0,\delta_2^{n+1}+\delta_2^{n}\big)+\sum_{k=1}^n\big(r^k,\tau^2(\delta_2^{n+1}+\delta_2^{n})\big).
\end{aligned}
\end{equation}
For the third term on the left-hand side of $(\ref{2.10})$, it has
\begin{equation}\label{2.11}
\begin{aligned}
     &a\left(\frac{\Delta_1^{n+1}-2\Delta_1^{n}+\Delta_1^{n-1}}{2},\Delta_1^{n+1}-\Delta_1^{n-1}\right)\\
     =&\frac{1}{2}a\left(\Delta_1^{n+1}-\Delta_1^{n}-\Delta_1^{n}+\Delta_1^{n-1},\Delta_1^{n+1}-\Delta_1^{n}+\Delta_1^{n}-\Delta_1^{n-1}\right)\\
     =&\frac{1}{2}\left(\|\Delta_1^{n+1}-\Delta_1^{n}\|_a^2-\|\Delta_1^{n}-\Delta_1^{n-1}\|_a^2\right)\\
     =&\frac{\tau^2}{2}\left(\|\delta_1^{n+1}\|_a^2-\|\delta_1^{n}\|_a^2\right).
\end{aligned}
\end{equation}
Substituting $(\ref{2.11})$ into $(\ref{2.10})$, and using another telescoping sum, we have
\begin{equation}\label{2.12}
\begin{aligned}
     &\|\delta_1^{n+1}\|_b^2-\|\delta_1^{1}\|_b^2+\|\delta_2^{n+1}\|_b^2-\|\delta_2^{1}\|_b^2\\
     &+\frac{\tau^2}{2}\sum_{k=1}^N\left(\|\delta_1^{k+1}\|_a^2-\|\delta_1^{k}\|_a^2\right)
     +a\big(\Delta^{n},\Delta^{n+1}\big)\\
     =&\sum_{k=1}^n\left[b\big(\delta_1^1-\delta_1^0,\delta_1^{n+1}+\delta_1^{n}\big)+\sum_{q=1}^k\big(r^q,\tau^2(\delta_1^{n+1}+\delta_1^{n})\big)+\sum_{q=1}^ka\big(s_1^q,\tau^2(\delta_1^{n+1}+\delta_1^{n})\big)\right]\\
     &+\sum_{k=1}^n\left[b\big(\delta_2^1-\delta_2^0,\delta_2^{n+1}+\delta_2^{n}\big)+\sum_{q=1}^k\big(r^q,\tau^2(\delta_2^{n+1}+\delta_2^{n})\big)\right].
\end{aligned}
\end{equation}
Each of the terms in $(\ref{2.12})$ can be estimated. For the last term on the left-hand side, by using \eqref{eq:CFL}, it has
\begin{equation}\label{2.13}
\begin{aligned}
     &a(\Delta^{n},\Delta^{n+1})\\
     =&\frac{1}{4}\left(a(\Delta^{n+1}+\Delta^{n},\Delta^{n+1}+\Delta^{n})-a(\Delta^{n+1}-\Delta^{n},\Delta^{n+1}-\Delta^{n})\right)\\
     =&a\left(\frac{\Delta^{n+1}+\Delta^{n}}{2},\frac{\Delta^{n+1}+\Delta^{n}}{2}\right)-\frac{\tau^2}{4}a(\delta^{n+1},\delta^{n+1})\\
     \geq&\left\|\frac{\Delta^{n+1}+\Delta^{n}}{2}\right\|_a^2-\frac{1}{2}\|\delta^{n+1}\|_b^2.
\end{aligned}
\end{equation}
For the second and the fourth terms on the left-hand side of $(\ref{2.12})$, we proceed with the standard procedure with the Cauchy-Schwarz inequality to see that
\begin{equation}\label{2.14}
\begin{aligned}
     &\|\delta_1^{1}\|_b^2+\|\delta_2^{1}\|_b^2\\
     =&\|\delta_1^0\|_b^2+\|\delta_2^0\|_b^2+b\big(\delta_1^1-\delta_1^0,\delta_1^1\big)
     +b\big(\delta_1^1-\delta_1^0,\delta_1^0\big)\\
     &+b\big(\delta_2^1-\delta_2^0,\delta_2^1\big)+b\big(\delta_2^1-\delta_2^0,\delta_2^0\big).
\end{aligned}
\end{equation}
Similarly, for the terms on the right-hand side, it has
\begin{equation}\label{2.15}
\begin{aligned}
     &b\big(\delta_1^1-\delta_1^0,\delta_1^{n+1}+\delta_1^{n}\big)+b\big(\delta_2^1-\delta_2^0,\delta_2^{n+1}+\delta_2^{n}\big)\\
     \leq& 2\left\|\delta_1^1-\delta_1^0\right\|_{L^2(\Omega)}\max_{0\leq k\leq N_T}\left\|\pi(\delta_1^k)\right\|_{L^2(\Omega)}\\
     &+2\left\|\delta_2^1-\delta_2^0\right\|_{L^2(\Omega)}\max_{0\leq k\leq N_T}\left\|\pi(\delta_2^k)\right\|_{L^2(\Omega)},
\end{aligned}
\end{equation}
and
\begin{equation}\label{2.16}
\begin{aligned}
     &\sum_{k=1}^n\sum_{q=1}^k\left(r^q,\delta^{k+1}+\delta^{k}\right)_{L^2(\Omega)}\\
     \leq&\sum_{k=1}^nR_1^k\left\|\pi(\delta^{k+1}+\delta^{k})\right\|_{L^2(\Omega)}+\sum_{k=1}^nR_2^k\left\|(I-\pi)(\delta^{k+1}+\delta^{k})\right\|_{L^2(\Omega)}\\
     \leq&2\left(\sum_{k=1}^nR_1^k\right)\max_{0\leq k\leq N_T}\left\|\pi(\delta^k)\right\|_{L^2(\Omega)}\\
     &+2\left(\sum_{k=1}^nR_2^k\right)\max_{0\leq k\leq N_{T}-1}\left\|(I-\pi)\left(\frac{\delta^{k+1}+\delta^{k}}{2}\right)\right\|_{L^2(\Omega)}.
\end{aligned}
\end{equation}
Finally, the rest parts on the right-hand side can be estimated as
\begin{equation}\label{2.17}
\begin{aligned}
&\sum_{k=1}^n\sum_{q=1}^ka\left(s_1^q,\tau^2(\delta_1^{n+1}+\delta_1^{n})\right)\\
\leq&\tau^2\sum_{k=1}^nR_3^k\|\delta_1^{n+1}+\delta_1^{n}\|_a\\
\leq&2\tau^2\left(\sum_{k=1}^nR_3^k\right)\max_{0\leq k\leq N_T}\|\delta_1^{k}\|_a.
\end{aligned}
\end{equation}

By using the initial condition, we have
\begin{equation}\label{2.18}
\begin{aligned}
N_{T}\|\delta_1^1-\delta_1^0\|_{L^2(\Omega)}+N_{T}\|\delta_2^1-\delta_2^0\|_{L^2(\Omega)}\leq C(\Lambda^{-1}H^2+\tau^2).
\end{aligned}
\end{equation}
Applying the estimates on $R_1^k,~R_2^k$ and $R_3^k$, we have
\begin{equation}\label{2.19}
\begin{aligned}
\tau^2\sum_{k=1}^{N_T-1}(R_1^k+R_2^k+R_3^k)\leq C(\Lambda^{-1}H^2+\tau^2).
\end{aligned}
\end{equation}
Using the triangle inequality with $\delta^0=\eta^0-\epsilon^0$, we have
\begin{equation}\label{2.20}
\begin{aligned}
&\max_{0\leq k\leq N_T}(\|\delta_1^k\|^2_{L^2(\Omega)}+\|\delta_2^k\|^2_{L^2(\Omega)})\\
&\leq C\left(\|\epsilon^0\|_{L^2(\Omega)}+\max_{0\leq k\leq N_T}\left(\|\eta_1^k\|_{L^2(\Omega)}+\|\eta_2^k\|_{L^2(\Omega)}\right)+\Lambda^{-1}H^2+\tau^2\right).
\end{aligned}
\end{equation}
Using another triangle inequality with $\delta^k=\eta^k-\epsilon^k$, we have
\begin{equation}\label{2.21}
\begin{aligned}
&\max_{0\leq k\leq N_T}\|\epsilon^k\|^2_{L^2(\Omega)}\\
&\leq C\left(\|\epsilon^0\|_{L^2(\Omega)}+\max_{0\leq k\leq N_T}\left(\|\eta_1^k\|_{L^2(\Omega)}+\|\eta_2^k\|_{L^2(\Omega)}\right)+\Lambda^{-1}H^2+\tau^2\right).
\end{aligned}
\end{equation}
Since $$\|\epsilon^0\|_{L^2(\Omega)}\leq C(\|\eta_1^0\|_{L^2(\Omega)}+\|\eta_2^0\|_{L^2(\Omega)} )\leq C\Lambda^{-1}H^2,$$
the proof can be completed.
\end{proof}
\bibliography{ref}

\end{document}